\UseRawInputEncoding
\documentclass[a4paper,11pt]{amsart}

\usepackage{amsmath,amssymb,amsthm}

\numberwithin{equation}{section} \theoremstyle{plain}
\newtheorem{thm}{Theorem}[section]

\newtheorem{prop}[thm]{Proposition}
\newtheorem{lem}[thm]{Lemma}
\newtheorem{cor}[thm]{Corollary}

\newtheorem{prob}[thm]{Problem}

\newtheorem{rem}[thm]{Remark}

\newtheorem{ack}{Acknowledgements}   
 \makeatletter

\DeclareMathOperator{\oi}{\mathbf{i}}
\DeclareMathOperator{\oj}{\mathbf{j}}
\DeclareMathOperator{\ok}{\mathbf{k}}

\DeclareMathOperator{\oRe}{\mathbf{Re}}

\DeclareMathOperator{\odiag}{\mathrm{diag}}

\DeclareMathOperator{\odist}{\mathrm{dist}}
\DeclareMathOperator{\oImag}{\mathrm{Image}}

\makeatother
\title[Isoparametric polynomials and sums of squares]{Isoparametric polynomials and sums of squares}
\author[J. Q. Ge]{Jianquan Ge}
\address{School of Mathematical Sciences, Beijing Normal
University, Beijing 100875, P. R. China}
\email{jqge@bnu.edu.cn}
\author[Z. Z. Tang]{Zizhou Tang$^{*}$}
\address{Chern Institute of Mathematics \& LPMC, Nankai University, Tianjin 300071, P. R. China}
\email{zztang@nankai.edu.cn}
\thanks {$^{*}$ the corresponding author.}
\subjclass[2010]{53C40; 14P99; 15A63.}
\date{}
\keywords{Hilbert's 17th problem, isoparametric polynomials, sum of squares, quadratic form.}
\thanks{The project is partially supported by Beijing NSF (Z190003), NSFC (No. 12171037, 11931007, 11871282), Nankai Zhide Foundation and Tianjin Outstanding Talents Foundation.}
\begin{document}
\maketitle


\begin{abstract}
Hilbert's 17th problem asks that whether every nonnegative polynomial can be a sum of squares of rational functions. It has been answered affirmatively by Artin. However, the question as to whether a given nonnegative polynomial is a sum of squares of polynomials is still a central question in real algebraic geometry. In this paper, we solve this question completely for the nonnegative polynomials associated with isoparametric polynomials, initiated by E. Cartan, which define the focal submanifolds of the corresponding isoparametric hypersurfaces.
\end{abstract}

\section{Introduction}\label{secintro}

A real polynomial $p(x)$ in $n$ variables is called \textit{positive semidefinite} (\emph{psd} for short) or \textit{nonnegative} if $p(x)\geq0$ for all $x\in\mathbb{R}^n$; it is called a \textit{sum of squares} (\emph{sos}) if there exist real polynomials $h_j$ such that $p=\sum h_j^2$. It is a central question in real algebraic geometry whether a given \emph{psd} polynomial is \emph{sos} (cf. \cite{Bl12,BSV16,BCR,Rez00,Rez07} and references therein).
As any \emph{psd} or \emph{sos} polynomial can be made homogeneous by adding an extra variable which preserves \emph{psd} or \emph{sos}, it is convenient to work with homogeneous polynomials (forms). Let $P_{n,d}$ and $\Sigma_{n,d}$ denote the sets of \emph{psd} and \emph{sos} forms in $n$ variables of even degree $d$, respectively. In this terminology, it is clear that $P_{n,d}\supseteq\Sigma_{n,d}$ and the question above asks whether or when a \emph{psd} form $p(x)\in P_{n,d}$ belongs to $\Sigma_{n,d}$.

The question above goes back to Minkowski's thesis defence in 1885.
It was Hilbert \cite{Hil1888} who showed that the equality $P_{n,d}=\Sigma_{n,d}$ holds only in the following four cases:
\begin{equation*}
P_{n,2}=\Sigma_{n,2}, \quad P_{1,d}=\Sigma_{1,d},\quad P_{2,d}=\Sigma_{2,d}, \quad P_{3,4}=\Sigma_{3,4}.
 \end{equation*}
 It follows that there exists a \emph{psd} but \emph{non-sos} form $p(x)\in P_{n,d}\backslash\Sigma_{n,d}$ if $n\geq3$ and $d\geq6$, or $n\geq4$ and $d\geq4$. Hilbert's proof used complex algebraic curves, and had no explicit example of a \emph{psd} polynomial that is \emph{non-sos}.
 77 years later, such an example was first constructed by Motzkin \cite{Mo65}. Since then many scattered examples were constructed by Robinson,
 Choi-Lam, Lax-Lax, Schm\"{u}dgen, and Reznick, etc. (cf. \cite{Rez00,Rez07}). Algorithms were also studied extensively and applied to many aspects like optimization theory, robotics and even self-driving cars (cf.\cite{AM17}).
  In the smallest cases: $(n,d)=(3,6)$ and $(4,4)$, Blekherman \cite{Bl12} first gave a complete unified geometric description of the difference between \emph{psd} and \emph{sos} forms.

   After the above remarkable theorem, Hilbert \cite{Hil1893} showed that any \emph{psd} form in $P_{3,d}$ ($d\geq6$) is a sum of squares of rational functions instead of polynomials. He then posed his famous Hilbert's 17th Problem in 1900 ICM: \textit{Must every \emph{psd} form be a sum of squares of rational functions} (\emph{sosr} for short)?  This was answered affirmatively by Artin \cite{Ar27} using orderings of fields. However, Artin's proof gives no specific \emph{sosr} representation of a \emph{psd }form. Uniform denominators $|x|^{2r}$ with sufficiently large $r$ were proved to exist for positive definite forms by P\'{o}lya and Reznick, i.e., if $p(x)\in P_{n,d}$ and $p(x)>0$ whenever $x\neq0$, then $|x|^{2r}p(x)\in\Sigma_{n,d}$. There are also many nonnegative, non-sos polynomials with
zeroes are known to become sos, after multiplying by $|x|^{2r}$. This holds for instance
for the Motzkin and Robinson polynomials (see Reznick \cite{Rez00}). On the other hand, there exist
nonnegative polynomials $f$ that will never become sum of squares after multiplying by $|x|^{2r}$
(and in fact by any positive polynomial). This is due to existence of
so-called ``bad points" (see Reznick \cite{Rez00}, page 16). Such a polynomial $f$ with bad points was given by Delzell
     $$D(w,x,y,z):=w^2(x^4y^2+y^4z^2+z^4x^2-3x^2y^2z^2)+z^8\in P_{4,8},$$
which has no \emph{sosr} representation with a uniform denominator $|x|^{2r}$ for any $r$.
One way to establish that a polynomial $f$ of degree $2d$ is not sos is to show that no polynomial of degree $d$ vanishes on the zero-set of $f$. This will be used in the proof of Theorems \ref{thmg422}, \ref{thm}, \ref{defcase}, \ref{g6m1} for some classes of isoparametric polynomials. However, this technique no longer applies after multiplication by $|x|^{2r}$.
For more history and developments we refer to the wonderful surveys \cite{Rez00, Rez07} by Reznick and \cite{BCR} by Bochnak-Coste-Roy.

In this paper we mainly consider the problem on a series of specific\emph{ psd} forms with significant geometric background, namely, isoparametric polynomials. It originated from the study of isoparametric hypersurfaces in unit spheres by E. Cartan \cite{Ca38} in the 1930s. Through a long history of efforts (e.g.,  \cite{Mun, OT75, FKM81, Ab83, DN85, Ta91, Fa99, St99, CCJ07, Im08, Miy13, Miy16, Chi11, Chi13, Chi16}, etc.), isoparametric hypersurfaces in unit spheres have been completely classified up to isometry. Equivalently, the isoparametric polynomials on Euclidean spaces have been completely classified up to orthogonal transformations.

A hypersurface of a Riemannian manifold is called \textit{isoparametric} if its nearby parallel hypersurfaces have constant mean curvature, or equivalently, it is locally a regular level set of an \textit{isoparametric function} $f$ (i.e., $|\nabla f|^2$ and $\Delta f$ are functions of $f$, cf. \cite{GT13, QT15}), or a regular leaf of an \textit{isoparametric foliation} (i.e., a singular Riemannian foliation of codimension $1$ with constant mean curvature regular leaves, cf. \cite{Ge16, GR15, Th10}). In unit spheres (or real space forms), Cartan showed that a hypersurface is isoparametric if and only if it has constant principal curvatures.

A fundamental result of M\"{u}nzner \cite{Mun} states that an
isoparametric hypersurface $M$ in a unit sphere $\mathbb{S}^{n-1}$ is an open part of
a level hypersurface of an isoparametric function $f=F|_{\mathbb{S}^{n-1}}$. Here $F$, called
a \emph{Cartan-M\"{u}nzner polynomial} (or \emph{isoparametric polynomial}),
is a homogeneous polynomial of degree $g$ on $\mathbb{R}^{n}$
satisfying the Cartan-M\"{u}nzner equation:
\begin{equation}\label{CMeq}
\left\{\begin{array}{ll}
|\nabla F|^2 =g^2|x|^{2g-2}, &\\
\Delta F=\frac{g^2}{2}(m_{-}-m_{+})|x|^{g-2},
  \end{array}
  \right. \quad x\in\mathbb{R}^{n},
\end{equation}
where $\nabla F$, $\Delta F$ denote the gradient and Laplacian of
$F$ on $\mathbb{R}^{n}$, respectively, $m_{\pm}$ denotes the
multiplicities of the maximal and minimal principal curvatures of
$M$ with respect to the normal direction $\frac{\nabla f}{|\nabla f|}$, and $g=\deg(F)$ is equal to the number of distinct principal curvatures of
$M$.

It is easy to see that $|\nabla f|^2=g^2(1-f^2)$ on the unit sphere $\mathbb{S}^{n-1}$, and thus $\oImag (f)=[-1,1]$, $f^{-1}(t)$, $t\in(-1,1)$, is a regular level set (thus an isoparametric hypersurface) and $f^{-1}(\pm1)=:M_{\pm}$ are smooth submanifolds, called \textit{focal submanifolds}, of codimension $m_{\pm}+1$ in $\mathbb{S}^{n-1}$.
In fact, given an isoparametric hypersurface $M$ in $\mathbb{S}^{n-1}$, it is clear that $M$ has exactly two focal submanifolds, say $M_{\pm}$. One then defines the corresponding isoparametric function $f$ on $\mathbb{S}^{n-1}$
by $f(x):=\cos(g\odist(x, M_+))$, where $\odist(x, M_+)$ is the spherically oriented distance from $x$ to the focal submanifold $M_+$ of $M$. We remark that if one takes
$M_-$ instead of $M_+$, the corresponding function becomes $-f$.
It turns out that the function
$$F(x):=|x|^gf(x/|x|)=|x|^g\cos(g\odist(x/|x|, M_+))$$
 is well-defined, and is exactly the corresponding Cartan-M\"{u}nzner polynomial on $\mathbb{R}^{n}$. For a systematic introduction of isoparametric theory, we refer to the excellent book by Cecil and Ryan \cite{CR15} and to a more updated survey by Chi \cite{Chi19}.

Using an elegant topological method, M\"{u}nzner \cite{Mun} proved the
remarkable result that the number $g$ must be $1$, $2$, $3$, $4$,
or $6$ (see a new simplified proof by Fang \cite{Fa17}). Now since $-1\leq f(x)\leq1$, we have $-|x|^g\leq F(x)\leq |x|^g$ on $\mathbb{R}^n$.
Thus we have infinitely many \emph{psd} forms $G_F^{\pm}$ and $H_F$ defined by
\begin{equation}\label{psdform}
\left\{\begin{array}{ll}
G_F^{\pm}(x):=|x|^g\pm F(x)\in P_{n,g} & g\;\; is\; even,\; g=2,4,6;\\
H_F(x):=|x|^{2g}-F(x)^2\in P_{n,2g}      & g=1,2,3,4,6.
\end{array}
\right.
\end{equation}
It is then natural to ask whether these explicit \emph{psd} forms (known to be nonnegative from their geometric background) are \emph{sos} or not. In this paper we solve this problem completely in accordance with the classification of isoparametric hypersurfaces in unit spheres. The proof will use representation theory of Clifford algebra for the cases when $g=4$, and deep geometric property of isoparametric hypersurfaces for the cases when $g=6$.


For $g=1$, isoparametric hypersurfaces are just hyperspheres $\mathbb{S}^{n-2}\subset\mathbb{S}^{n-1}$ and, up to a congruence, $F(x)=x_1$ is a coordinate function and thus $H_F$ is trivially \emph{sos}.

 Similarly, for $g=2$, they are the Clifford torus $\mathbb{S}^{k-1}\times \mathbb{S}^{n-k-1}\subset\mathbb{S}^{n-1}$ and, up to a congruence, $F(x)=\sum_{i=1}^kx_i^2-\sum_{i=k+1}^nx_i^2$ and thus $G_F^{\pm}$ are trivially \emph{sos}. 

 For $g=3$, Cartan showed that they are tubes around one of the four
 Veronese projective planes $\mathbb{FP}^2\subset\mathbb{S}^{3m+1}$ for $\mathbb{F}=\mathbb{R},\mathbb{C},\mathbb{H},\mathbb{O}$ with $m=1,2,4,8$. We show in Section \ref{secg3} that $H_F$ is always \emph{sos} with an explicit expression, not only for these four isoparametric polynomials with $g=3$ but also for $g=1,2,4,6$, simply by using the Cartan-M\"{u}nzner equation (\ref{CMeq}), Euler's formula and Lagrange's identity.

For $g=6$, there are only two classes of homogeneous isoparametric hypersurfaces in $\mathbb{S}^7$ and $\mathbb{S}^{13}$ with $m_+=m_-=:m=1,2$ respectively. We show in Section \ref{secg6} that for both isoparametric polynomials $F(x)$, neither of $G_F^{\pm}$  is \emph{sos}.

 The case $g=4$ is the most difficult case as in the classification process, because it is the only case in which there are infinitely many homogeneous and nonhomogeneous isoparametric hypersurfaces. Fortunately, due to the classification, we only need to consider the isoparametric polynomials of OT-FKM type and the exceptional two homogeneous cases with $(m_+,m_-)=(2,2), (4,5)$, which will be solved in Sections \ref{secg4OTFKM} and \ref{secg42245}, respectively. For both of these two types, we can always\footnote{For any isoparametric polynomial $F$ with multiplicities $(m_+, m_-)$, $F':=-F$ is an isoparametric polynomial with multiplicities $(m_-, m_+)$ determining the same class of isoparametric hypersurfaces with converse focal submanifolds $M_{\pm}'=M_{\mp}$. Hence we regard them as equivalent. The class $(4,5)$ would be replaced by $(5,4)$ for the sake of consistency (see Section \ref{secg42245}). }\label{footnote1} write $F(x)$ as $|x|^4$ minus some given \emph{sos} form (see (\ref{g422}, \ref{g445}, \ref{FKM isop. poly.})), thus $G_F^-$ is automatically \emph{sos}. However, it turns out that only in a few (though still infinitely many) classes $G_F^+$ is \emph{sos}.
  For the sake of clarity, we list the classification of these \emph{sos} forms in the following tables, where $k\in\mathbb{N}$, $(4,3)^I$ denotes the unique OT-FKM type with $(m_+,m_-)=(4,3)$ of the indefinite class, $G_F^{\pm}, H_F$ are \emph{psd} forms in (\ref{psdform}) with $F$ expressed in (\ref{Cartanpoly}, \ref{Cartanpoly248}), (\ref{g422}, \ref{g445}, \ref{FKM isop. poly.}), (\ref{isopg6}) for $g=3,4,6$ respectively.
\begin{table}[h]
\caption{Classification of \emph{sos} forms $G_F^{\pm}, H_F$ for $g= 1,2,3,6$} \label{table-1}
\centering
\begin{tabular}{|c|c|c|c|c|}
\hline
$g$ & 1 & 2 & 3&6\\
\hline
$G_F^+$ & -- & \emph{sos} & -- & \emph{non-sos}\\
\hline
$G_F^-$ & -- & \emph{sos} & -- & \emph{non-sos} \\
\hline
$H_F$ & \emph{sos}  & \emph{sos} & \emph{sos} & \emph{sos}\\
\hline
\end{tabular}
\end{table}
\begin{table}[h]
\caption{Classification of \emph{sos} forms $G_F^{\pm}, H_F$ for $g=4$} \label{table-2}
\centering
\begin{tabular}{|c|c|c|c|c|c|c|c|c|c|}
\hline
$(m_+,m_-)$ & $(2,2)$ & $(5,4)$ & $(1,k)$& $(2,2k-1)$ & $(3,4)$ & $(4,3)^I$ & $(5,2)$ &$(6,1)$ & others\\
\hline
$G_F^+$ & \emph{non-sos} & \emph{non-sos} & \emph{sos} & \emph{sos} & \emph{sos} & \emph{sos} & \emph{sos} & \emph{sos} & \emph{non-sos}\\
\hline
$G_F^-$ & \emph{sos} & \emph{sos} & \emph{sos} & \emph{sos} & \emph{sos} & \emph{sos} & \emph{sos} & \emph{sos} &\emph{sos}\\
\hline
$H_F$ & \emph{sos}  & \emph{sos} & \emph{sos} & \emph{sos} & \emph{sos} & \emph{sos} & \emph{sos} & \emph{sos} &\emph{sos}\\
\hline
\end{tabular}
\end{table}

For the \emph{non-sos} \emph{psd} forms $G_F^{\pm}$, we give them a simple and explicit \emph{sosr} expression with a uniform denominator $|x|^2$ for $g=4$ and $|x|^4$ for $g=6$ in Section \ref{secg3}. Note that these forms are not positive definite, as they have non-trivial zero sets ($\mathbb{R}M_{\mp}$ for $G_F^{\pm}$). These examples are the supplement to Artin's theorem on Hilbert's 17th problem, which is beyond the scope of P\'{o}lya and Reznick's theorem. Note also that $G_F^{\pm}$ have infinitely many zeroes and these non-sos psd forms have at least $8$ variables. This can be compared with a low dimensional rigidity result of Choi-Lam-Reznick \cite{CLR80} which shows that a \emph{psd} form in $P_{4,4}$ or $P_{3,6}$ with more than $11$ or $10$ projective zeroes must be \emph{sos}.

It needs to be emphasized, that the zeroes of $G_F^{\pm}$ are also of special importance because of their
 rich geometric properties as the focal submanifolds of isoparametric hypersurfaces in $\mathbb{S}^{n-1}$.
 For example, they are austere submanifolds (thus minimal) with constant principal curvatures independent of the choice of normal directions (cf. \cite{HL82, GTY18}).
 For the cases $g=4$ (resp. $g=6,  m_+=m_-=1$) with $(m_+,m_-)=(2,2)$, $(5,4)$, $(4k, l-4k-1)^D$ (OT-FKM type with $m_+\equiv0~(mod~ 4)$ of the definite class),
 we have shown a stronger result that, \textit{any quadratic form $($resp. cubic form$)$ vanishing on $(G_F^+)^{-1}(0)\cap\mathbb{S}^{n-1}=M_-$ $($resp. either of $M_{\pm}$$)$ is identically zero}, which implies the \emph{non-sos} property of $G_F^+$. In particular, the focal submanifold $M_-$ is not quadratic (resp. cubic). This answers partially an important question of Solomon \cite{So92}.
In fact, Solomon \cite{So92} had gotten the \emph{sos} cases of $G_F^{\pm}$ of Table \ref{table-2} (with $(3,4)$ and $(4,3)^I$ cases missing). He remarked that, the question as to whether both focal varieties might be quadratic seems difficult in general. This is important for estimates of eigenvalues and eigenfunctions of the Laplacian on isoparametric hypersurfaces, as Solomon showed that each quadratic form vanishing on one focal submanifold is an eigenfunction on every isoparametric hypersurface and the other focal submanifold in that family.

In Section \ref{sec-6}, besides further discussion on the zeroes of $G_F^+$ and the Solomon question, we provide some clearer formulae of the \emph{psd} forms $G_F^+$ for the isoparametric polynomials of OT-FKM type. For example, we get the interesting \emph{psd} forms $G_{km}$ for $m=1,2,3,4$ (see \ref{CS-non-sos1}, \ref{CS-non-sos2}, \ref{CS-sos3}, \ref{CS-sos4}), including an elementary \emph{non-sos} \emph{psd} form:
\begin{equation*}
G_{k4}(X,Y):=|X|^2|Y|^2-|\langle X,Y\rangle_{\mathbb{H}}|^2\in P_{8k,4}\setminus\Sigma_{8k,4}, \quad \textit{for}~X,Y\in \mathbb{H}^k,~k\geq2.
\end{equation*}
This immediately shows that the Cauchy-Schwarz inequality holds but Lagrange's identity does not hold for quaternions. By the \emph{sos} expression of $H_F$, we will also give an explicit \emph{sosr} expression of $G_{k4}$ with a uniform denominator (see the identity (\ref{sosr-Lag})), which generalizes Lagrange's identity for quaternions. Moreover, we will discuss some applications to orthogonal multiplications, and to the \emph{sos} problem on the Grassmannian $Gr_2(\mathbb{R}^l)$ that relates closely to the celebrated result of  Blekherman-Smith-Velasco \cite{BSV16} and to the \emph{sos} problem of Harvey-Lawson \cite{HL82}.

Though classified completely via many efforts, isoparametric hypersurfaces in unit spheres deserve to be even more attractive research objects. Because round spheres and Clifford tori are $g=1$ and $g=2$ isoparametric hypersurfaces with appropriate convexity, they are technically easier to be treated on many rigidity problems in geometric analysis than those with $g\geq3$. From this point of view, our study of the \emph{sos} problem on all isoparametric polynomials provides such an example of attempt. In particular, this algebraic study has various applications to geometry, e.g., as mentioned, to the Solomon question on eigenvalue's estimates. 

\section{General results from Cartan-M\"{u}nzner equation}\label{secg3}
In this section, we present some general results that can be easily deduced from the Cartan-M\"{u}nzner equation (\ref{CMeq}), including (i) that the \emph{psd} forms $H_F$ in (\ref{psdform}) are always \emph{sos}, and (ii) that the \emph{psd} forms $G_F^{\pm}$ in (\ref{psdform}) can be expressed as a sum of squares of rational functions with a uniform denominator $|x|^2$ for $g=4$ and $|x|^4$ for $g=6$. We also provide explicit formulae for the first nontrivial case when isoparametric polynomials are of degree $g=3$. Explicit formulae for $g=4,6$ will be provided in sections later.

Let $F(x)$ be an isoparametric polynomial of degree $g\in\{1,2,3,4,6\}$ and $H_F(x):=|x|^{2g}-F(x)^2$ be the \emph{psd} form as in (\ref{psdform}). We first show
\begin{prop}\label{prop-HF}
$H_F$ is \emph{sos}, i.e., a sum of squares of forms of degree $g$.
\end{prop}
\begin{proof}
As $F(x)$ is homogeneous of degree $g$, $\langle \nabla F(x),x\rangle=gF(x)$ by Euler's formula.
Then the conclusion follows directly from Lagrange's identity and the Cartan-M\"{u}nzner equation (\ref{CMeq}):
\begin{equation*}\label{lageq}
|\nabla F(x)\wedge x|^2=|\nabla F(x)|^2|x|^2-\langle \nabla F(x),x\rangle^2=g^2(|x|^{2g}-F(x)^2)=g^2H_F(x),
\end{equation*}
where $\wedge$ is the exterior product.
\end{proof}

For $g=2,4,6$, let $G_F^{\pm}(x):=|x|^g\pm F(x)$ be the \emph{psd} forms as in (\ref{psdform}). We have
\begin{prop}\label{prop-sosrx24}
For even $g$, $|x|^{g-2}G_F^{\pm}(x)$ is \emph{sos}, i.e., a sum of squares of forms of degree $g-1$.
\end{prop}
\begin{proof}
Taking the gradient of $G_F^{\pm}(x)$, we have
$$\nabla G_F^{\pm}(x)=g|x|^{g-2}x\pm \nabla F(x).$$
Using Euler's formula and the Cartan-M\"{u}nzner equation (\ref{CMeq}), we get
\begin{equation*}
|\nabla G_F^{\pm}(x)|^2=g^2|x|^{2g-2}+|\nabla F(x)|^2\pm 2g |x|^{g-2}\langle x, \nabla F(x)\rangle
=2g^2|x|^{g-2}(|x|^g\pm F(x)).
\end{equation*}
Thus $|x|^{g-2}G_F^{\pm}(x)=|\nabla G_F^{\pm}(x)|^2/2g^2$ is \emph{sos}.
\end{proof}

As introduced in Section \ref{secintro}, Cartan classified isoparametric hypersurfaces in unit spheres with $g=3$, showing that they are tubes around one of the four Veronese projective planes $M_{\pm}\cong\mathbb{FP}^2\subset\mathbb{S}^{3m+1}$ for $\mathbb{F}=\mathbb{R},\mathbb{C},\mathbb{H},\mathbb{O}$ with $m_+=m_-=m=1,2,4,8$.
Cartan's isoparametric polynomial on $\mathbb{R}^5$ with $g=3, m=1$ is defined by
\begin{equation}\label{Cartanpoly}
F_C(x)=x_0^3+\frac{3}{2}x_0(x_2^2+x_3^2-2x_4^2-2x_1^2)   +\frac{3\sqrt{3}}{2}x_1(x_2^2-x_3^2)+3\sqrt{3} x_2x_3x_4,
\end{equation}
for $x=(x_0,\cdots,x_4)\in \mathbb{R}^5$, which will be used in Section \ref{secg6} for the case $g=6, m=1$.
The other three Cartan polynomials on $\mathbb{R}^{8}, \mathbb{R}^{14}, \mathbb{R}^{26}$ with $g=3, m=2,4,8$ can be defined similarly as
\begin{equation}\label{Cartanpoly248}
\begin{aligned}
F_C(x_0,x_1,X_2,X_3,X_4)&=
x_0^3+\frac{3}{2}x_0\Big(|X_2|^2+|X_3|^2-2|X_4|^2-2x_1^2\Big)  \\
&\quad  +\frac{3\sqrt{3}}{2}x_1\Big(|X_2|^2-|X_3|^2\Big)+3\sqrt{3}\oRe(X_2X_3X_4),
\end{aligned}
\end{equation}
where $x_0,x_1\in\mathbb{R}$, $X_2,X_3,X_4\in\mathbb{C},\mathbb{H},\mathbb{O}$ for $m=2,4,8$, respectively, and $\oRe$ denotes the real part.
Note that the two focal submanifolds $M_{\pm}=F_C^{-1}(\pm1)\cap \mathbb{S}^{3m+1}\cong \mathbb{FP}^2$ are antipodal to each other in the sphere and their union $M_+\cup M_-$ is exactly the set of spherical zeroes of the \emph{sos} form $H_{F_C}\in P_{3m+2, 6}$ in (\ref{psdform}). Hence $M_+\cup M_-$ is a cubic variety but separately neither of $M_{\pm}$ is cubic.  
This is different from the case of $g=4$, where $M_+\cup M_-$ is a quartic variety as zeroes of $H_F$ but always non-quadratic as shown by Solomon \cite{So92}.
Moreover, $M_+$ is always quadratic as zeroes of the \emph{sos} quartic form $G_F^-$ while $M_-$ is often non-quadratic as we will show in the following sections.

\section{On isoparametric with $g=4$, $(m_+,m_-)=(2,2),(5,4)$}\label{secg42245}
In this section, for the two exceptional homogeneous isoparametric hypersurfaces with $g=4$, $(m_+,m_-)=(2,2),(5,4)$ in $\mathbb{S}^9$ and $\mathbb{S}^{19}$, respectively, we prove that any quadratic form vanishing on $M_-$ is identically zero, which implies that $M_-$ is non-quadratic and the \emph{psd} form $G_F^+$ of (\ref{psdform}) is \emph{non-sos}. According to Solomon \cite{So92}, the corresponding isoparametric polynomials $F(x)$ are given by (\ref{g422}) and (\ref{g445}) below, which immediately shows that $G_F^-$ is \emph{sos} and $M_+$ is quadratic in both cases.

Before the proof, we first prepare two lemmas. This part treats with polynomials in $\mathbb{R}^5$.
Consider the Horn form $h$ and the Choi-Lam quartic \emph{non-sos} \emph{psd} form $H$ \cite{CL77}:
$$\begin{aligned}
& h(x_1,\cdots,x_5)=(x_1+\cdots +x_5)^2-4(x_1x_2+x_2x_3+x_3x_4+x_4x_5+x_5x_1), \\
& H(x_1,\cdots,x_5)=h(x_1^2,\cdots,x_5^2).
\end{aligned}$$
Denote by $\mathcal{Z}$ the spherical zeroes of $H$, that is,  $\mathcal{Z}=\{x\in \mathbb{S}^{4}\mid H(x)=0\}$.
\begin{lem}\label{lem}
$\mathcal{Z}$ is a union of ten circles. More precisely,

$\mathcal{Z}=S_1^{\pm} \cup S_2^{\pm} \cup S_3^{\pm} \cup S_4^{\pm} \cup S_5^{\pm} $,
where
 $S_1^{\pm}=\{(\pm\frac{1}{\sqrt{2}},a,0,0,b)\mid a^2+b^2=\frac{1}{2} \} $,\\
 $S_2^{\pm}=\{(b,\pm\frac{1}{\sqrt{2}},a,0,0)\mid a^2+b^2=\frac{1}{2} \} $,
 $S_3^{\pm}=\{(0,b,\pm\frac{1}{\sqrt{2}},a,0)\mid a^2+b^2=\frac{1}{2} \} $,\\
 $S_4^{\pm}=\{(0,0,b,\pm\frac{1}{\sqrt{2}},a)\mid a^2+b^2=\frac{1}{2} \} $,
 $S_5^{\pm}=\{(a,0,0,b,\pm\frac{1}{\sqrt{2}})\mid a^2+b^2=\frac{1}{2} \} $.
\end{lem}
\begin{proof} Observe that there are two equalities
\begin{equation*}
\begin{aligned}
h=(x_1-x_2+x_3+x_4-x_5)^2+4x_2x_4+4x_3(x_5-x_4),  \\
h=(x_1-x_2+x_3-x_4+x_5)^2+4x_2x_5+4x_1(x_4-x_5).
\end{aligned}
\end{equation*}
It follows that $H(x)\geq 0$ for $x^2_5\geq x_4^2$ by the first equality, and
$H(x)\geq 0$ for $x^2_5 \leq x_4^2$ by the second equality. Thus, $H$ is indeed a \emph{psd} form. To determine the spherical zeroes,
we consider two cases when $x^2_5\geq x_4^2$ and when $x^2_5\leq x_4^2$,
and make use of the two equalities mentioned above respectively. The determination of $\mathcal{Z}$ is complicated but elementary
and will be omitted.
\end{proof}

\begin{lem}\label{lem}\label{lemg422}
Any quadratic form $P$ vanishing on $\mathcal{Z}$ is identically zero. In particular, $\mathcal{Z}$ is not quadratic.
\end{lem}
\begin{proof}
Suppose that $P=P(x_1,\cdots, x_5)$ is a
quadratic form vanishing on $\mathcal{Z}$.
We observe that $P(x_1,x_2,x_3,0,0)$ vanishing on $S_2^{\pm}$ is of the form
$\lambda(x_1^2-x_2^2+x_3^2) $ with $\lambda$ being a real number. Similarly, the corresponding conclusions hold for $P(0,x_2,x_3,x_4,0)$,  $P(0,0,x_3,x_4,x_5)$, $P(x_1,0,0,x_4,x_5)$ and $P(x_1,x_2,0,0,x_5)$ on
 $S_3^{\pm}$, $S_4^{\pm}$, $S_5^{\pm}$ and $S_1^{\pm}$, respectively.
 Clearly, these imply that $P$ is identically zero.
 \end{proof}

Let us now consider the isoparametric polynomial $F$ with $g=4$, $(m_+,m_-)=(2,2)$. According to Solomon \cite{So92}, $F$ comes from
the map
$$ |(X\wedge X)^*|^2\;\quad \;\textit{for} \;X\in \Lambda^2(\mathbb{R}^5)\cong \mathbb{R}^{10}, $$
where $\wedge$ is the exterior product, and $*$ is the Hodge star operator $* : \Lambda^4(\mathbb{R}^5) \longrightarrow \Lambda^1(\mathbb{R}^5)\cong \mathbb{R}^5. $
Choose an oriented orthonormal basis $\{e_1,\cdots,e_5\}$ in $\mathbb{R}^5$.
Represent
 $$\begin{aligned}X&=x_1e_1\wedge e_2 + x_2e_1\wedge e_3 + x_3e_1\wedge e_4 + x_4e_1\wedge e_5 +x_5e_2\wedge e_3\\
    &\quad +x_6e_2\wedge e_4 + x_7e_2\wedge e_5 + x_8e_3\wedge e_4 + x_9e_3\wedge e_5 +x_{10}e_4\wedge e_5. \end{aligned}$$
It is clear that
$$\begin{aligned}
&\frac{1}{2}(X\wedge X)^*= (x_5x_{10}-x_6x_9+x_7x_8)e_1   +   (-x_2x_{10}+x_3x_9-x_4x_8)e_2 \\
&\quad\quad +(x_1x_{10}-x_3x_7+x_4x_6)e_3  +   (-x_1x_9+x_2x_7-x_4x_5)e_4  + (x_1x_8-x_2x_6+x_3x_5)e_5. \end{aligned}$$
The corresponding isoparametric polynomial $F(X)=|X|^4-2|(X\wedge X)^*|^2$ is
\begin{equation}\label{g422}
\begin{aligned}
   F(x_1,\cdots, x_{10})&=(x_1^2+\cdots+ x_{10}^2)^2-8\Big\{ (x_5x_{10}-x_6x_9+x_7x_8)^2 \\
     &\quad  +   (-x_2x_{10}+x_3x_9-x_4x_8)^2 +(x_1x_{10}-x_3x_7+x_4x_6)^2\\
     &\quad  +   (-x_1x_9+x_2x_7-x_4x_5)^2  + (x_1x_8-x_2x_6+x_3x_5)^2 \Big\}
 \end{aligned}
 \end{equation}
Clearly, the focal submanifold $M_+=F^{-1}(+1)\cap \mathbb{S}^9\cong Gr_2(\mathbb{R}^5)$ is quadratic and $G_F^-$ of (\ref{psdform}) is \emph{sos}. For the other focal submanifold $M_-=F^{-1}(-1)\cap \mathbb{S}^9\cong \mathbb{CP}^3$,
we have
\begin{thm}\label{thmg422}
Any quadratic form $Q=Q(x_1, \cdots,x_{10})$ vanishing on $M_-$ is identically zero. In particular, $M_-$ is not quadratic and $G_F^+$ of $(\ref{psdform})$ on $\mathbb{R}^{10}$ is \emph{non-sos}.
\end{thm}

\begin{proof} For $x=(x_1,\cdots,x_{10})$,
let $G(x)=G_F^+(x)/2=(|x|^4+F(x))/2$, namely,
$$\begin{aligned}
G(x)&=|x|^4-4\Big\{ (x_5x_{10}-x_6x_9+x_7x_8)^2   +   (-x_2x_{10}+x_3x_9-x_4x_8)^2    \\
&\quad +(x_1x_{10}-x_3x_7+x_4x_6)^2  +   (-x_1x_9+x_2x_7-x_4x_5)^2  + (x_1x_8-x_2x_6+x_3x_5)^2 \Big\}.
 \end{aligned}$$
It is clear that $M_-=\{x\in \mathbb{S}^9 \mid G(x)=0\}$. We prove in the following steps. For convenience, denote by $\mathbb{R}^5_{ijklm}$ the $5$-space with coordinates $x_i$, $x_j$, $x_k$, $x_l$, $x_m$.

(I). Restricting $x\in \mathbb{R}^{10}$ to $\mathbb{R}^5_{23679}$,
one gets
$$\begin{aligned}
&\quad G(0,x_2,x_3,0,0,x_6,x_7,0,x_9,0)\\
&= (x_2^2+x_6^2+x_9^2+x_3^2+x_7^2)^2 -4 (x_2^2x_6^2+x_6^2x_9^2+x_9^2x_3^2+x_3^2x_7^2+x_7^2x_2^2)\\
&=H(x_2,x_6,x_9,x_3,x_7).
 \end{aligned}$$
On the other hand, restricting $x\in \mathbb{R}^{10}$ to $\mathbb{R}^5_{1458,10}$,
one gets
$$\begin{aligned}
&\quad G(x_1,0,0,x_4,x_5,0,0,x_8,0,x_{10})\\
&= (x_1^2+x_{10}^2+x_5^2+x_4^2+x_8^2)^2 -4 (x_1^2x_{10}^2+x_{10}^2x_5^2+x_5^2x_4^2+x_4^2x_8^2+x_8^2x_1^2)\\
&=H(x_1,x_{10},x_5,x_4,x_8).
 \end{aligned}$$
Suppose now that a quadratic form $Q=Q(x_1, \cdots,x_{10})$ vanishes on $M_-$.
Applying Lemma \ref{lemg422}, we see that $Q=Q(x)$ is a bilinear form on $\mathbb{R}^5_{23679}\times \mathbb{R}^5_{1458,10}$.

(II). Let us restrict $x\in \mathbb{R}^{10}$ to $\mathbb{R}^5_{2456,10}$. Then
$$\begin{aligned}
&\quad G(0,x_2,0,x_4,x_5,x_6,0,0,0,x_{10})\\
&= (x_2^2+x_6^2+x_4^2+x_5^2+x_{10}^2)^2 -4 (x_2^2x_6^2+x_6^2x_4^2+x_4^2x_5^2+x_5^2x_{10}^2+x_{10}^2x_2^2)\\
&=H(x_2,x_6,x_4,x_5,x_{10}).
 \end{aligned}$$
On the other hand,  restricting $x\in \mathbb{R}^{10}$ to $\mathbb{R}^5_{13789}$,
one gets
$$\begin{aligned}
&\quad G(x_1,0,x_3,0,0,0,x_7,x_8,x_9,0)\\
&= (x_1^2+x_8^2+x_7^2+x_3^2+x_9^2)^2 -4 (x_1^2x_8^2+x_8^2x_7^2+x_7^2x_3^2+x_3^2x_9^2+x_9^2x_1^2)\\
&=H(x_1,x_8,x_7,x_3,x_9).
 \end{aligned}$$

Applying Lemma \ref{lemg422}, and summarizing the arguments above, we can write $Q$ as
$$\begin{aligned}
&\quad Q(x_1,\cdots,x_{10}) \\
& =a_1x_1x_2+ a_2x_1x_6 +  a_3x_2x_8  + a_4 x_3x_4  +  a_5 x_3x_5   + a_6x_3x_{10} + a_7x_4x_7 \\
&\quad   + a_8x_4x_9  +a_9 x_5x_7 +a_{10} x_5x_9  + a_{11}x_6x_8 + a_{12}x_7x_{10} + a_{13}x_9x_{10},
 \end{aligned}$$
with real numbers $a_1,\cdots, a_{13}$.

(III). Let us restrict $x\in \mathbb{R}^{10}$ to $\mathbb{R}^5_{34569}$. Then
$$\begin{aligned}
&\quad G(0,0,x_3,x_4,x_5,x_6,0,0,x_9,0)\\
&= (x_3^2+x_5^2+x_4^2+x_6^2+x_9^2)^2 -4 (x_3^2x_5^2+x_5^2x_4^2+x_4^2x_6^2+x_6^2x_{9}^2+x_{9}^2x_3^2)\\
&=H(x_3,x_5,x_4,x_6,x_9).
 \end{aligned}$$
On the other hand,  restricting $x\in \mathbb{R}^{10}$ to $\mathbb{R}^5_{1278,10}$,
one gets
$$\begin{aligned}
&\quad G(x_1,x_2,0,0,0,0,x_7,x_8,0,x_{10})\\
&= (x_1^2+x_8^2+x_7^2+x_2^2+x_{10}^2)^2 -4 (x_1^2x_8^2+x_8^2x_7^2+x_7^2x_2^2+x_2^2x_{10}^2+x_{10}^2x_1^2)\\
&=H(x_1,x_8,x_7,x_2,x_{10}). \end{aligned}$$

Applying Lemma \ref{lemg422}, and summarizing the arguments above, we deduce
$$\begin{aligned}
&\quad Q(x_1,\cdots,x_{10})\\
&=a_2x_1x_6 + a_6x_3x_{10} + a_7x_4x_7 +a_9 x_5x_7  + a_{11}x_6x_8 + a_{13}x_9x_{10}. \end{aligned}$$

(IV). Let us restrict $x\in \mathbb{R}^{10}$ to $\mathbb{R}^5_{24678}$. Then
$$\begin{aligned}
&\quad G(0,x_2,0,x_4,0,x_6,x_7,x_8,0,0)\\
&= (x_2^2+x_4^2+x_6^2+x_8^2+x_7^2)^2 -4 (x_2^2x_6^2+x_6^2x_4^2+x_4^2x_8^2+x_8^2x_7^2+x_7^2x_2^2)\\
&=H(x_2,x_6,x_4,x_8,x_7).\end{aligned}$$
On the other hand,  restricting $x\in \mathbb{R}^{10}$ to $\mathbb{R}^5_{1359,10}$,
one gets
$$\begin{aligned}
&\quad G(x_1,0,x_3,0,x_5,0,0,0,x_9,x_{10})\\
&= (x_1^2+x_9^2+x_3^2+x_5^2+x_{10}^2)^2 -4 (x_1^2x_9^2+x_9^2x_3^2+x_3^2x_5^2+x_5^2x_{10}^2+x_{10}^2x_1^2)\\
&=H(x_1,x_9,x_3,x_5,x_{10}).\end{aligned}$$

 Applying Lemma \ref{lemg422}, and summarizing the arguments above, we deduce
$$ Q(x_1,\cdots,x_{10})=  a_2x_1x_6 +a_9 x_5x_7 .$$

(V). We observe that $M_-$ contains the set
 $$\{(0,0,x_3,0,x_5,0,x_7,0,0,0)\mid x_3^2=x_5^2+x_7^2=\frac{1}{2}\},$$
 and thus the assumption that $Q(x)$ vanishes on $M_-$ implies $Q(x)=a_2x_1x_6$.
At last, we observe that $M_-$ contains the set
 $$\{(x_1,0,0,0,0,x_6,0,0,x_9,0)\mid x_9^2=x_1^2+x_6^2=\frac{1}{2}\},$$
 and thus the assumption that $Q(x)$ vanishes on $M_-$ implies $Q(x)\equiv 0$.
 \end{proof}

Now we turn to the isoparametric polynomial with $g=4$, $(m_+,m_-)=(4,5)$. In fact, we consider the equivalent version (see the footnote \ref{footnote1}) with $(m_+,m_-)=(5,4)$ for the sake of consistency. According to Solomon \cite{So92},
it comes from the map
$$ |(Z\wedge Z)^*|^2\; \;\textit{for} \;Z\in \Lambda^2(\mathbb{C}^5)\cong \mathbb{C}^{10}\cong \mathbb{R}^{20}. $$
Choose an oriented orthonormal basis $\{e_1,\cdots,e_5\}$ in $\mathbb{C}^5$.
Represent
$$\begin{aligned}
Z&=z_1e_1\wedge e_2 + z_2e_1\wedge e_3 + z_3e_1\wedge e_4 + z_4e_1\wedge e_5 +z_5e_2\wedge e_3\\
&\quad +z_6e_2\wedge e_4 + z_7e_2\wedge e_5 + z_8e_3\wedge e_4 + z_9e_3\wedge e_5 +z_{10}e_4\wedge e_5.\end{aligned}$$

It is clear that
the isoparametric polynomial $F'$ defined by

$$F'(Z)=|Z|^4-2 |(Z\wedge Z)^*|^2$$
is equal to
\begin{equation}\label{g445}
\begin{aligned}
 F'(z_1,\cdots, z_{10})&=\Big(|z_1|^2+\cdots+ |z_{10}|^2\Big)^2-8\Big\{ |(z_5z_{10}-z_6z_9+z_7z_8)|^2\\
 &\quad   +   |(-z_2z_{10}+z_3z_9-z_4z_8)|^2    +|(z_1z_{10}-z_3z_7+z_4z_6)|^2 \\
 &\quad   +   |(-z_1z_9+z_2z_7-z_4z_5)|^2  + |(z_1z_8-z_2z_6+z_3z_5)|^2 \Big\},
\end{aligned}
\end{equation}
with $z_j=x_j+\sqrt{-1}y_j$, $j=1,\cdots,10$.
It is clear that the focal submanifold $M_+^{13}=(F')^{-1}(+1)\cap\mathbb{S}^{19}$ is quadratic  and $G_{F'}^-$ of (\ref{psdform}) is \emph{sos}. As before, we define $G'$ by
 $G'(z)=G_{F'}^+/2=(F'(z)+|z|^4)/2$, namely,
$$ G'(z_1,\cdots, z_{10})=\Big(|z_1|^2+\cdots+ |z_{10}|^2\Big)^2-4\Big\{ |(z_5z_{10}-z_6z_9+z_7z_8)|^2   +   |(-z_2z_{10}+z_3z_9-z_4z_8)|^2    $$
$$ +|(z_1z_{10}-z_3z_7+z_4z_6)|^2  +   |(-z_1z_9+z_2z_7-z_4z_5)|^2  + |(z_1z_8-z_2z_6+z_3z_5)|^2 \Big\}. $$
It is clear that $G'(z)\geq 0$, and $M^{14}_-=\{z\in S^{19}\mid G'(z)=0\}.$
   Taking $y_1=\cdots=y_{10}=0$, this isoparametric polynomial $F'$ with multiplicities $(5,4)$ becomes $F$ in (\ref{g422}), replacing $z_j$ by $x_j$.
   As a consequence, we get from Theorem \ref{thmg422}:
\begin{cor}\label{cor}
The \emph{psd} form $G_{F'}^+=2G'$ of $(\ref{psdform})$ on $\mathbb{R}^{20}$ is \emph{non-sos}.
\end{cor}
Furthermore, we can show
\begin{thm}\label{thm}
Any quadratic form $Q$ on $\mathbb{R}^{20}$ vanishing on $M^{14}_-$ is identically zero. In particular $M^{14}_-$ is not quadratic.
\end{thm}

\begin{proof} Suppose that a quadratic form $Q$ vanishes on $M_-$.
At first, let us take $y_1=\cdots=y_{10}=0$, or take $x_1=\cdots=x_{10}=0$.
Applying Theorem \ref{thmg422}, we see that $Q$ is a bilinear form on $\{x_1,\cdots, x_{10}\}$ and $\{y_1,\cdots,y_{10}\}$.
Namely, $Q=\sum a_{ij}x_iy_j$, with $a_{ij}\in \mathbb{R}$ and $i,j=1,\cdots,10$.

Next, let us take $x_i=y_i$ for $i=1,\cdots,10$. Then,
   $$ \frac{1}{4}G'(x_1,\cdots,x_{10},x_1,\cdots, x_{10}) =G(x_1,\cdots, x_{10}).              $$
By the assumption, $\sum a_{ij}x_ix_j$ vanishes on the spherical zeroes of $G$. Applying Theorem \ref{thmg422} again, we
see that $a_{ij}=-a_{ji}$ , for $i,j=1,\cdots,10$.

Now for $i<j$, considering the value of $Q$ at the zero point $z$ of $G'$ with
$$z_i=x_i=\frac{1}{\sqrt{2}},\;z_j=\sqrt{-1}y_j=\frac{\sqrt{-1}}{\sqrt{2}}$$
 and all other $z_k=0$ where $z_iz_j$ appears in  $|(Z\wedge Z)^*|^2$ of $G'$, we see that $a_{ij}=0$ for $(i,j)=(5,10)$, $(6,9)$, $(7,8)$, $(2,10)$, $(3,9)$, $(4,8)$, $(1,10)$, $(3,7)$, $(4,6)$, $(1,9)$, $(2,7)$, $(4,5)$, $(1,8)$, $(2,6)$, or $(3,5)$.

 For any other pair $(i,j)$ with $i<j$, we can also show that $a_{ij}=0$. For simplicity, without loss of generality we take $(i,j)=(1,2)$ for example. Since there are items $z_1z_{10}$ and $z_2z_{10}$ in $|(Z\wedge Z)^*|^2$, we have one zero point $Z$ of $G'$ with $z_1=x_1=\frac{1}{2}$, $z_2=\sqrt{-1}y_2=\frac{\sqrt{-1}}{2}$, $z_{10}=\sqrt{-1}y_{10}=\frac{\sqrt{-1}}{\sqrt{2}}$ and all other $z_k=0$. Thus
 $$ 0=Q(Z)=\frac{1}{4}a_{12}+\frac{1}{2\sqrt{2}}a_{1,10}=\frac{1}{4}a_{12},$$
as we have shown $a_{1,10}=0$.
 \end{proof}

\section{On isoparametric of OT-FKM type with $g=4$}\label{secg4OTFKM}
In this section, we classify the classes of isoparametric polynomials of hypersurfaces of OT-FKM type with $g=4$ such that $G_F^+=|x|^g+ F(x)$ of (\ref{psdform}) is \emph{sos}.

Recall that an OT-FKM type isoparametric polynomial is defined as (cf. \cite{OT75, FKM81})
\begin{equation}\label{FKM isop. poly.}
F(x) = |x|^4 - 2\displaystyle\sum_{\alpha = 0}^{m}{\langle
P_{\alpha}x,x\rangle^2},  \quad x\in\mathbb{R}^{2l},
\end{equation}
where $\{P_0,\cdots,P_m\}$ is  a symmetric Clifford system on $\mathbb{R}^{2l}$, i.e.,
$P_{\alpha}$'s are symmetric matrices satisfying $P_{\alpha}P_{\beta}+P_{\beta}P_{\alpha}=2\delta_{\alpha\beta}I_{2l}$.
Then the multiplicity pair is $(m_+, m_-)=(m, l-m-1)$. Two Clifford systems $\{P_0,\cdots,P_m\}$ and $\{Q_0,\cdots,Q_m\}$ on $\mathbb{R}^{2l}$ are called
algebraically equivalent if there exists $A\in O(\mathbb{R}^{2l})$ such that $Q_\alpha = AP_\alpha A^t$
for all $\alpha \in\{0, \cdots, m\}$. They are called geometrically equivalent when there exists $B\in O(\mathrm{Span}\{P_0, \cdots,P_m\})$ such that $\{Q_0,\cdots,Q_m\}$ and $\{B(P_0),\cdots,B(P_m)\}$ are algebraically equivalent, which give two isoparametric polynomials that are congruent under an orthogonal transformation of $\mathbb{R}^{2l}$. We will apply representation theory of Clifford algebra to prove Theorem \ref{SOS-FKM} below.

As introduced in Section \ref{secintro}, we can define the \emph{psd} forms $G_F^{\pm}\in P_{2l,4}$ as (\ref{psdform}).
Clearly, $G_F^-=|x|^g-F(x)$ is \emph{sos} and the focal submanifold $M_+=F^{-1}(1)\cap \mathbb{S}^{2l-1}$ is quadratic and defined as
\begin{equation}\label{def-M+}
M_+^{m_++2m_-}=\{x\in\mathbb{S}^{2l-1}\mid \langle P_\alpha x,x\rangle=0, \alpha=0,\cdots,m\}.
\end{equation}
From now on, we write $G_F=G_F^+/2$ for simplicity. Then
\begin{equation}\label{nonnegativepolyG}
G_F(x)=(F(x)+|x|^4)/2=|x|^4-\sum_{\alpha=0}^m\langle P_\alpha x,x\rangle^2.
\end{equation}

We are concerned with whether the other focal submanifold $M_-=F^{-1}(-1)\cap \mathbb{S}^{2l-1}$ is quadratic or not. Note that $M_-$ is just the set of spherical zeroes of $G_F$ and can be expressed as
\begin{equation}\label{def-M-}
M_-^{2m_++m_-}=G_F^{-1}(0)\cap\mathbb{S}^{2l-1}=\{x\in\mathbb{S}^{2l-1}\mid Px=x, \textit{ for some }P\in\Sigma\},
\end{equation}
where $\Sigma=\{P\in \mathrm{Span}\{P_0,\cdots,P_m\}\mid |P|^2=\mathrm{tr}(PP^t)=2l\}$ is the Clifford sphere (see \cite{FKM81}). Therefore, if the \emph{psd} form $G_F$ is a sum of squares of quadratic forms, then $M_-$ is obviously quadratic. It turns out that for almost all cases $G_F$ is \emph{non-sos}.
\begin{thm}\label{SOS-FKM}
The \emph{psd} form $G_F$ in $(\ref{nonnegativepolyG})$ on $\mathbb{R}^{2l}$, associated with the OT-FKM type, is \emph{sos} if and only if $m=1,~2$, or $(m_+, m_-)=(m, l-m-1)=(5,2),~(6,1), ~(3,4)$ or $(4,3)$ of the indefinite class.
\end{thm}
\begin{proof}
We first show the sufficiency.
For $m=1,~2$, Solomon \cite{So92} had proven that $G_F(x)$ in (\ref{nonnegativepolyG}) is a sum of squares of quadratic forms. We repeat the proof for the sake of completeness.
In these cases, $l=km$, $(m_+, m_-)=(1, k-2)$, for any integer $k\geq3$ for $m=1$; or $(m_+, m_-)=(2, 2k-3)$, for any integer $k\geq 2$ for $m=2$.
The coordinate $x\in \mathbb{R}^{2l}$ can be written as $x=(X,Y)\in \mathbb{F}^k\oplus \mathbb{F}^k$ where $\mathbb{F}=\mathbb{R}$ for $m=1$ and $\mathbb{F}=\mathbb{C}$ for $m=2$,
respectively.
Without loss of generality, we can write the Clifford system $\{P_0,\cdots,P_m\}$ in matrix form as (\ref{Cliffordsys-alg}) below where $E_1$ corresponds to the complex structure on $\mathbb{R}^l\cong \mathbb{C}^k$ in the case of $m=2$. Then the isoparametric polynomial $F(x)$ can be written as
$$F(X,Y)=(|X|^2+|Y|^2)^2-2\Big((|X|^2-|Y|^2)^2+4|\langle X,Y\rangle|^2\Big),$$
where $\langle \cdot,\cdot\rangle$ denotes the Hermitian inner product if $m=2$.
Using Lagrange's identity
$$|X|^2|Y|^2=|X\wedge Y|^2+|\langle X,Y\rangle|^2,$$
where $X\wedge Y\in\Lambda^2(\mathbb{F}^k)$ is the exterior product,
we obtain the \emph{sos}-expression of Solomon:
\begin{equation}\label{SOS-FKM1}
G_F(X,Y)=4\Big(|X|^2|Y|^2-|\langle X,Y\rangle|^2\Big)=4|X\wedge Y|^2=4\sum_{1\leq i<j\leq k}|X_iY_j-X_jY_i|^2,
\end{equation}
for any  $X=(X_1,\cdots,X_k),~ Y=(Y_1,\cdots,Y_k)\in \mathbb{F}^k$.

For the cases of $(m_+, m_-)=(5,2),~(6,1), ~(3,4)$ or $(4,3)$ of the indefinite class (for the definition see Subsection \ref{subsec-indefinite} below), the isoparametric foliations are just those OT-FKM type isoparametric foliations with converse multiplicities $(m_-,m_+)$. Correspondingly the \emph{psd} forms $G_F$ can be expressed as a sum of squares of quadratic forms (see \cite{FKM81}). We repeat the proof for the sake of completeness.
The isoparametric polynomials $F(x)$ in these cases can be defined in the following unified way.
Let $\{P_0,\cdots,P_8\}$ be the Clifford system on $\mathbb{R}^{16}$ corresponding to the Clifford algebra $\{E_1,\cdots,E_7\}$ on the Octonions $\mathbb{R}^8$ (see also (\ref{Cliffordbasis}) below).
This Clifford system has the following property
$$\sum_{\alpha=0}^8\langle P_\alpha x,x\rangle^2=|x|^4, \quad x\in \mathbb{R}^{16}.$$
Therefore, taking $m=5,6,3,4$ respectively, the corresponding isoparametric polynomials $F(x)$ can be defined as
$$F(x)=|x|^4-2\sum_{\alpha=0}^m\langle P_\alpha x,x\rangle^2.$$
Thus
\begin{equation}\label{SOS-FKM2}
G_F(x)=|x|^4-\sum_{\alpha=0}^m\langle P_\alpha x,x\rangle^2=\sum_{\alpha={m+1}}^8\langle P_\alpha x,x\rangle^2, \quad x\in \mathbb{R}^{16}.
\end{equation}
We remark that for $m=1,2$, the formula (\ref{SOS-FKM2}) gives another \emph{sos} expression of $G_F$ different from those in (\ref{SOS-FKM1}) (remarked also by  Solomon in \cite{So92}).

To prove the necessity, we only need to show for all other cases that the nonnegative polynomial $G_F(x)$ is not a sum of squares of quadratic forms. We will show this case-by-case in the following subsections.
\end{proof}

\subsection{$m\equiv 0 ~(mod ~4)$, definite case.}
A Clifford system $\{P_0,\cdots,P_m\}$ is called definite if $P_0\cdots P_m=\pm I_{2l}$ (for indefinite case see Subsection \ref{subsec-indefinite} below). In fact, we obtain the following stronger result.
\begin{thm}\label{defcase}
For the case of $m\equiv0 ~(mod ~4)$, when the Clifford system $\{P_0,\cdots,P_m\}$ is definite, any quadratic form $Q$ vanishing on $M_-$ is identically zero. In particular, $M_-$ is not quadratic, and the \emph{psd} form $G_F$ in $(\ref{nonnegativepolyG})$ is \emph{non-sos}.
\end{thm}
\begin{proof}
Without loss of generality, we can write the Clifford system $\{P_0,\cdots,P_m\}$ in matrix form under the decomposition $\mathbb{R}^{2l}=E_+(P_0)\oplus E_-(P_0)\cong\mathbb{R}^l\oplus \mathbb{R}^l$, where $E_{\pm}(P_0)$ are the eigenspaces of the eigenvalues $\pm1$ of $P_0$, by
\begin{equation}\label{Cliffordsys-alg}
P_0=\begin{pmatrix}
    I_l &  0 \\
     0 & -I_l
\end{pmatrix}, \quad P_1=\begin{pmatrix}
  0 & I_l    \\
 I_l & 0
\end{pmatrix}, \quad P_{\alpha+1}=\begin{pmatrix}
  0 & E_\alpha    \\
 -E_\alpha & 0
\end{pmatrix}, \quad 1\leq\alpha\leq m-1,
\end{equation}
where $\{E_1,\cdots,E_{m-1}\}$ generates a Clifford algebra on $\mathbb{R}^l$, i.e., $E_\alpha$'s are skew-symmetric matrices satisfying $E_\alpha E_\beta+E_\beta E_\alpha=-2\delta_{\alpha\beta}I_l$.

Recall (\ref{def-M-}) that $E_{\pm}(P_0)\cap\mathbb{S}^{2l-1}\subset M_-$. Thus any quadratic form $Q(x)=\sum_{i,j=1}^{2l}q_{ij}x_ix_j$ vanishing on $M_-$ can be expressed in matrix form $Q=(q_{ij})$ by
\begin{equation}\label{quadformQ}
Q=\begin{pmatrix}
  0 & B    \\
 B^t & 0
\end{pmatrix}.
\end{equation} Moreover, since for any $u\in\mathbb{R}^l$, $x=(u,u)\in E_+(P_1)$, and by (\ref{def-M-}), $E_+(P_1)\cap\mathbb{S}^{2l-1}\subset M_-$, we have
$$0=Q(x)=2\langle Bu,u\rangle,$$
which implies that $B$ is skew-symmetric.

Similarly, for each $\alpha=1,\cdots,m-1$, for any $u\in\mathbb{R}^l$, $x=(u,-E_{\alpha}u)\in E_+(P_{\alpha+1})$, and by (\ref{def-M-}), $E_+(P_{\alpha+1})\cap\mathbb{S}^{2l-1}\subset M_-$, the equalities
$$0=Q(x)=2\langle Bu,E_{\alpha}u\rangle$$
imply
\begin{equation}\label{anti-comm}
BE_{\alpha}=-E_{\alpha}B, \quad \alpha=1,\cdots,m-1.
\end{equation}

Now for the case of $m\equiv0 ~(mod ~4)$, when $\{P_0,\cdots,P_m\}$ is definite, i.e.,
$$P_0\cdots P_m=\begin{pmatrix}
    E_1\cdots E_{m-1} &  0 \\
     0 & E_1\cdots E_{m-1}
\end{pmatrix}=\pm I_{2l},$$
 it follows from (\ref{anti-comm}) that
\begin{equation}\label{anti-comm2}
B E_1\cdots E_{m-1}=(-1)^{m-1} E_1\cdots E_{m-1}B,
\end{equation}
which implies that $B=-B$ as $m$ is even and $ E_1\cdots E_{m-1}=\pm I_l$, and thus $B=0$.
\end{proof}

\subsection{$m\equiv 0 ~(mod ~4)$, $(m_+,m_-)\neq(4,3)$, indefinite cases.}\label{subsec-indefinite}
With the same notations as in the last subsection, a Clifford system $\{P_0,\cdots,P_m\}$ is called indefinite if $$P_0\cdots P_m=\begin{pmatrix}
    E_1\cdots E_{m-1} &  0 \\
     0 & E_1\cdots E_{m-1}
\end{pmatrix}\neq\pm I_{2l}.$$

 Recall from \cite{FKM81} that each Clifford system is algebraically equivalent to a direct sum of irreducible Clifford systems. An irreducible
Clifford system $\{P_0,\cdots,P_m\}$ on $\mathbb{R}^{2l}$ exists precisely for the following values of $m$ and $l=\delta(m)$ in Table \ref{table3}:
\begin{table}[h!]
\caption{Dimension $\delta(m)$ of irreducible representation of Clifford algebra}\label{table3}
\centering
\begin{tabular}{|c|c|c|c|c|c|c|c|c|c|}
\hline
$m$ & $1$ & $2$ & $3$ & $4$ & $5$ & $6$ & $7$ & $8$ & $\cdots~ m+8$\\
\hline
$\delta(m)$& $1$ &$2$ &$4$ &$4$ & $8$ & $8$ & $8$ & $8$ & $\cdots~ 16\delta(m)$\\
\hline
\end{tabular}
\end{table}

 For $m\equiv 0 ~(mod ~4)$, there are exactly two algebraic equivalence classes of irreducible Clifford systems $\{P_0,\cdots,P_m\}$ on $\mathbb{R}^{2\delta(m)}$ (resp. Clifford algebras $\{E_1,\cdots,E_{m-1}\}$ on $\mathbb{R}^{\delta(m)}$), distinguished from each other by the choice of sign in
 $P_0\cdots P_m=\pm I_{2\delta(m)}$ (resp. $E_1\cdots E_{m-1}=\pm I_{\delta(m)}$). They are geometrically equivalent because replacing $P_m$ by $-P_m$ swaps them.  If one constructs all direct sums of both of the irreducible algebraic classes with altogether $k$ summands, then the invariant
$|\mathrm{tr}(P_0\cdots P_m)|$ (invariant under geometric equivalence) takes on $1+[\frac{k}{2}]$ different values. Thus there are exactly one definite and $[\frac{k}{2}]$ indefinite geometric equivalence classes of Clifford systems on $\mathbb{R}^{2l}$ with $l=k\delta(m)$. For indefinite case, $k\geq2$ is necessary and the system is reducible.

With these investigations prepared, we are ready to show
\begin{thm}\label{NonSOS-indef}
For the case of $m\equiv0 ~(mod ~4)$, $(m_+,m_-)\neq(4,3)$, when the Clifford system $(P_0,\cdots,P_m)$ is indefinite,  the \emph{psd} form $G_F$ in $(\ref{nonnegativepolyG})$ is \emph{non-sos}.
\end{thm}
\begin{proof}
We prove it by contradiction. Assume there are quadratic forms $Q_1,\cdots,Q_N$ such that
\begin{equation}\label{SOS-ind-assum}
\sum_{i=1}^NQ_i(x)^2=G_F(x)=|x|^4-\sum_{\alpha=0}^m\langle P_\alpha x,x\rangle^2.
\end{equation}
Then each quadratic form $Q_i$ vanishes on $M_-$, which implies that $Q_i$'s are in the same form as in (\ref{quadformQ}) and (\ref{anti-comm}), i.e.,
\begin{equation}\label{SOS-Qiform}
Q_i=\begin{pmatrix}
  0 & B_i    \\
 B_i^t & 0
\end{pmatrix}, \quad B_iE_{\alpha}=-E_{\alpha}B_i, \quad \alpha=1,\cdots,m-1,
\end{equation}
where each $B_i$ is skew-symmetric. Recalling the representation (\ref{Cliffordsys-alg}) of the Clifford system, the equation (\ref{SOS-ind-assum}) is now equivalent to
the identity
\begin{eqnarray}
&&\sum_{i=1}^N\langle B_iu, v\rangle^2=\frac{1}{4}G_F(x) \nonumber \\
&=&\frac{1}{4}\Big(|(u,v)|^4-\langle(u,-v), (u,v)\rangle^2-\langle(v,u), (u,v)\rangle^2-4\sum_{\alpha=1}^{m-1}\langle E_\alpha u, v\rangle^2\Big)\nonumber \\
&=&|u|^2|v|^2-\langle u,v\rangle^2-\sum_{\alpha=1}^{m-1}\langle E_\alpha u,v\rangle^2, \quad\quad\quad\quad\quad x=(u,v)\in\mathbb{R}^l\oplus\mathbb{R}^l.\label{SOS-ind-assum2}
\end{eqnarray}
This identity involves much information. For example, it follows
\begin{equation}\label{SOS-ind-assum3}
\sum_{i=1}^N |B_i|^2=l^2-lm,
\end{equation}
where 
$$\sum_{i=1}^N |B_i|^2=\sum_{p,q=1}^l\sum_{i=1}^N\langle B_ie_p, e_q\rangle^2,$$
for an orthonormal basis $\{e_p\}$ of $\mathbb{R}^l$. This formula holds independent of reducibility.

On the other hand, we consider the decomposition of $\{P_0,\cdots,P_m\}$ on $\mathbb{R}^{2l}$ with $l=k\delta(m)$ into a direct sum of $k\geq2$ irreducible Clifford systems on $\mathbb{R}^{2\delta(m)}$ (denoted with a superscript $r=1,\cdots,k$) so that
\begin{equation}\label{Clifforddecom-irr}
\begin{array}{cccc}
\mathbb{R}^{2l}=&\mathbb{R}^{2\delta(m)} & \oplus  \cdots \oplus & \mathbb{R}^{2\delta(m)}  \\
(P_0,\cdots,P_m)=&(P_0^1,\cdots,P_m^1) & \oplus \cdots  \oplus & (P_0^k,\cdots,P_m^k).
\end{array}
\end{equation}
Here the irreducible Clifford systems $\{P_0^r,\cdots,P_m^r\}$ on $\mathbb{R}^{2\delta(m)}$ can be expressed in the form as (\ref{Cliffordsys-alg}) so that
\begin{equation}\label{Cliffordsys-alg-irr}
P_0^r=\begin{pmatrix}
    I_{\delta(m)} &  0 \\
     0 & -I_{\delta(m)}
\end{pmatrix}, \quad P_1^r=\begin{pmatrix}
  0 & I_{\delta(m)}     \\
 I_{\delta(m)}  & 0
\end{pmatrix}, \quad P_{\alpha+1}^r=\begin{pmatrix}
  0 & E_\alpha^r    \\
 -E_\alpha^r & 0
\end{pmatrix},  \\
~~
\end{equation}
$\alpha=1,\cdots,m-1,$
where $\{E_1^r,\cdots,E_{m-1}^r\}$ generates an irreducible Clifford algebra on each $\mathbb{R}^{\delta(m)}$ of the decomposition of $\{E_1,\cdots,E_{m-1}\}$ on $\mathbb{R}^{l}=\mathbb{R}^{\delta(m)}  \oplus  \cdots \oplus  \mathbb{R}^{\delta(m)}$.

Now for the case of $m\equiv0 ~(mod ~4)$, when $\{P_0,\cdots,P_m\}$ is indefinite ($[\frac{k}{2}]$ classes in total), according to the statements at the beginning of this subsection, without loss of generality,
there is some $[\frac{k+1}{2}]\leq r_0<k$ such that
\begin{equation}\label{Clifforddecom-irr2}
P_0^r\cdots P_m^r=\begin{pmatrix}
    E_1^r\cdots E_{m-1}^r &  0 \\
     0 & E_1^r\cdots E_{m-1}^r
\end{pmatrix}=\left\{
\begin{array}{ll}
I_{2\delta(m)} & r\leq r_0 \\
-I_{2\delta(m)} & r>r_0.
\end{array}
\right.
\end{equation}
Furthermore, we can set
\begin{equation}\label{Clifforddecom-irr-indef}
\begin{array}{ll}
E_\alpha^1=\cdots=E_\alpha^k & \textit{for } \alpha=1,\cdots,m-2,\\
E_{m-1}^1=\cdots=E_{m-1}^{r_0}=-E_{m-1}^{r_0+1}=\cdots=-E_{m-1}^k.
\end{array}
\end{equation}
Therefore, as proved in the last subsection, any quadratic form $Q$ vanishing on $M_-$ is identically zero when restricted to each irreducible component $\mathbb{R}^{2\delta(m)}$ due to (\ref{anti-comm}), (\ref{anti-comm2}) and (\ref{Clifforddecom-irr2}).
Regarding $B_i$ as skew-symmetric operator on $\mathbb{R}^l$, we can rewrite $B_i$ with respect to the irreducible decomposition (\ref{Clifforddecom-irr}) as
\begin{equation}\label{Clifforddecom-irr3}
\begin{array}{cc}
B_i: &\mathbb{R}^{l}=\mathbb{R}^{\delta(m)}  \oplus  \cdots \oplus  \mathbb{R}^{\delta(m)}\rightarrow \mathbb{R}^{\delta(m)}  \oplus  \cdots \oplus  \mathbb{R}^{\delta(m)}=\mathbb{R}^{l} \\
&B_i=\begin{pmatrix}
   B_i^{11}&  \cdots & B_i^{1k} \\
   \vdots  &  \ddots & \vdots \\
   B_i^{k1} & \cdots & B_i^{kk}
\end{pmatrix},
\end{array}
\end{equation}
where each $B_i^{rr}=0$ since it acts as the restriction of $Q_i$ to the $r$-th component $\mathbb{R}^{2\delta(m)}\subset\mathbb{R}^{2l}$, and $(B_i^{rs})^t=-B_i^{sr}:\mathbb{R}^{\delta(m)}\rightarrow \mathbb{R}^{\delta(m)}$ as $B_i$ is skew-symmetric. Let $\{e_a^r\}_{a=1}^{\delta(m)}$ be an orthonormal basis of the $r$-th component $\mathbb{R}^{\delta(m)}\subset\mathbb{R}^{l}$.
Since $E_\alpha=E_\alpha^1\oplus\cdots\oplus E_\alpha^k$ acts orthogonally, it follows from (\ref{SOS-ind-assum2}) that for $r\neq s$,
\begin{equation}\label{Brs}
\sum_{i=1}^N |B_i^{rs}|^2=\sum_{a,b=1}^{\delta(m)}\Big(|e_a^r|^2|e_b^s|^2-\langle e_a^r, e_b^s\rangle^2-\sum_{\alpha=1}^{m-1}\langle E_\alpha e_a^r, e_b^s \rangle^2\Big)=\delta(m)^2,
\end{equation}
and thus
\begin{equation}\label{SOS-ind-assum3'}
\sum_{i=1}^N |B_i|^2=\sum_{r,s=1}^k\sum_{i=1}^N |B_i^{rs}|^2=\sum_{r\neq s}\sum_{i=1}^N |B_i^{rs}|^2=(k^2-k)\delta(m)^2=l^2-l\delta(m).
\end{equation}
Then comparing (\ref{SOS-ind-assum3}) and (\ref{SOS-ind-assum3'}) we obtain $m=\delta(m)$, which implies immediately that $m=4$ or $m=8$.

Let us consider now the case when $k\geq3$. Noting that
$$E_1\cdots E_{m-1}=E_1^1\cdots E_{m-1}^1\oplus\cdots\oplus E_1^k\cdots E_{m-1}^k=\begin{pmatrix}
    I_{r_0\delta(m)} &  0 \\
     0 & -I_{(k-r_0)\delta(m)}
\end{pmatrix},$$
we deduce from (\ref{anti-comm2}) in the same way as before that for each $i$,
\begin{equation}\label{qua00}
B_i^{rs}=0 \quad \textit{for either }~~ 1\leq r,s\leq r_0, ~~ \textit{or } ~~ r_0+1\leq r,s\leq k.
\end{equation}
This contradicts (\ref{Brs}) if $k\geq3$. In fact, similar arguments as in (\ref{SOS-ind-assum3'}) imply
\begin{equation}\label{SOS-ind-assum3''}
\sum_{i=1}^N |B_i|^2=2\sum_{r\leq r_0<s}\sum_{i=1}^N |B_i^{rs}|^2=\Big(k^2-r_0^2-(k-r_0)^2\Big)\delta(m)^2.
\end{equation}
This contradicts  (\ref{SOS-ind-assum3}) and (\ref{SOS-ind-assum3'}), since $\delta(m)\geq m$, $r_0^2\geq r_0$, $(k-r_0)^2\geq k-r_0$, and $r_0^2+(k-r_0)^2>r_0+(k-r_0)=k$ if $k\geq3$.
Hence we are only left with considering the case when $k=2$, namely, the indefinite classes $(m_+,m_-)=(4,3)$ and $(8,7)$.

The indefinite class $(m_+,m_-)=(4,3)$ has been excluded in the assumption. We deduce a contradiction for the last class $(m_+,m_-)=(8,7)$ as follows.
Firstly it follows from (\ref{qua00}) that in this case (\ref{Clifforddecom-irr3}) becomes
\begin{equation}\label{Clifforddecom-irr3''}
\begin{array}{cc}
B_i: &\mathbb{R}^{l}=\mathbb{R}^{8}  \oplus  \mathbb{R}^{8}\rightarrow \mathbb{R}^{8}  \oplus   \mathbb{R}^{8}=\mathbb{R}^{l} \\
&B_i=\begin{pmatrix}
  0 & C_i \\
  -C_i^t & 0
\end{pmatrix},
\end{array}
\end{equation}
where $C_i=B_i^{12}:~  \mathbb{R}^{8}\rightarrow \mathbb{R}^{8}$.
The equation (\ref{SOS-ind-assum2}) becomes
\begin{equation}\label{SOS-ind-assum2'}
\begin{array}{ll}
&\sum\limits_{i=1}^N\Big(\langle C_i v^2, u^1\rangle-\langle C_i u^2, v^1\rangle\Big)^2=(|u^1|^2+|u^2|^2)(|v^1|^2+|v^2|^2)\\
&\quad\quad\quad \quad\quad\quad-\Big(\langle u^1,v^1\rangle+\langle u^2,v^2\rangle\Big)^2 -\sum\limits_{\alpha=1}^7\Big(\langle E_\alpha^1u^1,v^1\rangle+\langle E_\alpha^2u^2,v^2\rangle\Big)^2,
\end{array}
\end{equation}
for any $u=(u^1,u^2), ~v=(v^1,v^2)\in\mathbb{R}^8\oplus \mathbb{R}^8$. Restricting to $u^2=v^2=0$ (or $u^1=v^1=0$), we have
\begin{equation}\label{Cliffordbasis}
\langle u^1,v^1\rangle^2+\sum\limits_{\alpha=1}^7\langle E_\alpha^ru^1,v^1\rangle^2=|u^1|^2|v^1|^2, \quad (u^1,v^1)\in\mathbb{R}^8\oplus \mathbb{R}^8, \quad r=1,2,
\end{equation}
which is also trivially implied by the Clifford algebra on the Octonions $\mathbb{R}^8$. Then the equation (\ref{SOS-ind-assum2'}) becomes
 \begin{equation}\label{SOS-ind-assum2''}
\begin{array}{ll}
&\sum\limits_{i=1}^N\Big(\langle C_i v^2, u^1\rangle-\langle C_i u^2, v^1\rangle\Big)^2=|u^1|^2|v^2|^2+|u^2|^2|v^1|^2\\
&\quad\quad\quad \quad\quad\quad-2\langle u^1,v^1\rangle\langle u^2,v^2\rangle -2\sum\limits_{\alpha=1}^7\langle E_\alpha^1u^1,v^1\rangle\langle E_\alpha^2u^2,v^2\rangle.
\end{array}
\end{equation}
Taking $(v^1,v^2)=(-u^1,u^2)$ in (\ref{SOS-ind-assum2''}), we obtain
\begin{equation}\label{SOS-ind-assum2'''}
\sum_{i=1}^N\langle C_iu^2, u^1\rangle^2=|u^1|^2|u^2|^2,\quad u=(u^1,u^2)\in\mathbb{R}^8\oplus \mathbb{R}^8.
\end{equation}
Thus by canceling the equalities in the form (\ref{SOS-ind-assum2'''}), the equation (\ref{SOS-ind-assum2''}) becomes
 \begin{equation}\label{SOS-ind-assum2-4}
\sum\limits_{i=1}^N\langle C_i v^2, u^1\rangle\langle C_i u^2, v^1\rangle=
\langle u^1,v^1\rangle\langle u^2,v^2\rangle +\sum\limits_{\alpha=1}^7\langle E_\alpha^1u^1,v^1\rangle\langle E_\alpha^2u^2,v^2\rangle,
\end{equation}
for any $u=(u^1,u^2), ~v=(v^1,v^2)\in\mathbb{R}^8\oplus \mathbb{R}^8$.
Taking $(v^1,v^2)=(u^2,u^1)$ in (\ref{SOS-ind-assum2-4}) and using (\ref{Clifforddecom-irr-indef}), we find
 \begin{equation}\label{SOS-ind-assum2-5}
\sum\limits_{i=1}^N\langle C_i u^1, u^1\rangle\langle C_i u^2, u^2\rangle=
\langle u^1,u^2\rangle^2 + \langle E_7^1u^1,u^2\rangle^2-\sum\limits_{\alpha=1}^6\langle E_\alpha^1u^1,u^2\rangle^2.
\end{equation}
For any fixed unit vector $u^1\in\mathbb{R}^8$, it follows from (\ref{Cliffordbasis}) that $\{u^1, E_1^1u^1,\cdots,E_7^1u^1\}$ constitutes an orthonormal basis of $\mathbb{R}^8$.
Taking contraction of (\ref{SOS-ind-assum2-5}) with respect to this basis for $u^2$, we obtain
 \begin{equation}\label{SOS-ind-assum2-6}
\sum\limits_{i=1}^N\langle C_i u^1, u^1\rangle \mathrm{tr}(C_i)=-4|u^1|^2.
\end{equation}
Lastly, taking contraction of (\ref{SOS-ind-assum2-6}) for $u^1$, we obtain the following contradiction
\begin{equation*}
\sum\limits_{i=1}^N \Big(\mathrm{tr}(C_i)\Big)^2=-32. \qedhere
\end{equation*}
\end{proof}

\subsection{$m\equiv 3 ~(mod ~4)$, $(m_+,m_-)\neq(3,4)$}
 Note that for all cases of $m\not\equiv 0~(mod~4)$, there exists exactly one geometric equivalence class of Clifford systems $\{P_0,\cdots,P_m\}$ on $\mathbb{R}^{2l}$ with $l=k\delta(m)$ and thus exactly one congruence class of isoparametric polynomials (cf. \cite{FKM81}). Using the techniques in previous subsections and the representation theory of Clifford algebra for this case, we can show
\begin{thm}\label{NonSOS-mod4=3}
For the case of $m\equiv3 ~(mod ~4)$, $(m_+,m_-)\neq(3,4)$, the \emph{psd} form $G_F$ in $(\ref{nonnegativepolyG})$ is \emph{non-sos}.
\end{thm}
\begin{proof}
As in the last subsection, we prove it by contradiction. Assume there are quadratic forms $Q_1,\cdots,Q_N$ such that (\ref{SOS-ind-assum}) holds, i.e.,
\begin{equation*}
\sum_{i=1}^NQ_i(x)^2=G_F(x)=|x|^4-\sum_{\alpha=0}^m\langle P_\alpha x,x\rangle^2.
\end{equation*}
We still have the formulae (\ref{Cliffordsys-alg}-\ref{anti-comm}) and (\ref{SOS-ind-assum}-\ref{SOS-ind-assum3}). For this case, there always exists a skew-symmetric operator $E_m\in O(l)$ such that $\{E_1,\cdots,E_{m-1},E_m\}$ generates a Clifford algebra on $\mathbb{R}^l$ of definite class, i.e., $E_1\cdots E_m=I_l$, corresponding to a Clifford system $\{P_0,\cdots,P_m,P_{m+1}\}$ on $\mathbb{R}^{2l}$. It follows from (\ref{SOS-Qiform}) that $E_m=-E_1\cdots E_{m-1}$ commutes with each $B_i$ as $m-1$ is even, i.e., $A_i:=B_iE_m=E_mB_i$ is symmetric. Taking $v=E_mu$, the equation (\ref{SOS-ind-assum2}) gives
\begin{equation}\label{SOS-mod4=3-assum}
\sum_{i=1}^N\langle B_iu, E_mu\rangle^2=\sum_{i=1}^N\langle A_iu, u\rangle^2=|u|^4, \quad u\in\mathbb{R}^l.
\end{equation}
Taking Hessian of both sides of the equation (\ref{SOS-mod4=3-assum}), it follows
\begin{equation}\label{SOS-mod4=3-assum2}
\sum_{i=1}^N \Big(2A_iuu^tA_i+\langle A_iu,u\rangle A_i\Big)=2uu^t+|u|^2I_l, \quad u\in\mathbb{R}^l,
\end{equation}
where $u$ is regarded as a column vector in $\mathbb{R}^l$.
Taking trace of (\ref{SOS-mod4=3-assum2}), we see
\begin{equation*}
\sum_{i=1}^N \Big(2u^tA_i^2u+\langle A_iu,u\rangle \mathrm{tr}(A_i)\Big)=(2+l)|u|^2, \quad u\in\mathbb{R}^l,
\end{equation*}
which is equivalent to
\begin{equation}\label{SOS-mod4=3-assum3}
\sum_{i=1}^N \Big(2A_i^2+\mathrm{tr}(A_i)A_i\Big)=(2+l)I_l.
\end{equation}
Noting that by (\ref{SOS-ind-assum3}), $$\mathrm{tr}(\sum_{i=1}^N 2A_i^2)=2\sum_{i=1}^N |B_i|^2=2(l^2-lm),$$
we obtain from (\ref{SOS-mod4=3-assum3}) that
$$\sum_{i=1}^N \Big(\mathrm{tr}(A_i)\Big)^2=l(2+2m-l)\geq0,$$
which holds if and only if each $\mathrm{tr}(A_i)=0$ and $m=3$ or $7$ (when $m\equiv 3~(mod~4)$), as $l=k\delta(m)$ increases much more quickly than $m$.
 Hence we are only left with considering the cases $(3,4)$ and $(7,8)$, where the $(3,4)$ case has been excluded in the assumption.

 We deduce a contradiction for the last case $(m_+,m_-)=(7,8)$ as follows. In this case, $l=16$, $k=2$, $\delta(m)=8$. According to the representation theory of Clifford algebra, we know there also exists a skew-symmetric operator $\widetilde{E}_m\in O(l)$ such that $\{E_1,\cdots,E_{m-1},\widetilde{E}_m\}$ generates a Clifford algebra on $\mathbb{R}^l$ of indefinite class, i.e., $E_1\cdots E_{m-1} \widetilde{E}_m\neq I_l$. The difference between $E_m$ and $\widetilde{E}_m$ can be shown by their irreducible decompositions as (\ref{Clifforddecom-irr-indef}), i.e.,
 \begin{equation}\label{Clifforddecom-78}
\begin{array}{ccc}
 &\mathbb{R}^{l}=\mathbb{R}^{8}  \oplus  \mathbb{R}^{8}\rightarrow \mathbb{R}^{8}  \oplus   \mathbb{R}^{8}=\mathbb{R}^{l}& \\
&E_m=\begin{pmatrix}
  E_m^1 & 0 \\
  0 & E_m^1
\end{pmatrix},&\\
&\widetilde{E}_m=\begin{pmatrix}
  E_m^1 & 0 \\
  0 & -E_m^1
\end{pmatrix},&
\end{array}
\end{equation}
where $E_m^1=-E_1^1\cdots E_{m-1}^1\in O(8)$. With respect to this decomposition, we rewrite the skew-symmetric operator $B_i$ as (\ref{Clifforddecom-irr3}), i.e.,
 \begin{equation*}
\begin{array}{cc}
 &\mathbb{R}^{l}=\mathbb{R}^{8}  \oplus  \mathbb{R}^{8}\rightarrow \mathbb{R}^{8}  \oplus   \mathbb{R}^{8}=\mathbb{R}^{l} \\
&B_i=\begin{pmatrix}
  B_i^{11} & B_i^{12} \\
  -(B_i^{12})^t & B_i^{22}
\end{pmatrix},
\end{array}
\end{equation*}
where $B_i^{11}$ and $B_i^{22}$ are skew-symmetric. Since $B_iE_m=E_mB_i$,
we have $$B_i^{11}E_m^1=E_m^1B_i^{11}, \quad B_i^{22}E_m^1=E_m^1B_i^{22}, \quad B_i^{12}E_m^1=E_m^1B_i^{12}.$$
Setting $\widetilde{A}_i:=(B_i\widetilde{E}_m+\widetilde{E}_mB_i)/2$, we derive from the identities above and (\ref{Clifforddecom-78}) that
\begin{equation}\label{SOS-mod4=3-assum78A}
\widetilde{A}_i=\begin{pmatrix}
  B_i^{11}E_m^1 & 0 \\
  0 & -B_i^{22}E_m^1
\end{pmatrix}=:\begin{pmatrix}
  \widetilde{A}_i^1 & 0 \\
  0 & \widetilde{A}_i^2
\end{pmatrix}.
\end{equation}
Taking $v=\widetilde{E}_mu$ in (\ref{SOS-ind-assum2}), we obtain the following formula similar to (\ref{SOS-mod4=3-assum})
\begin{equation}\label{SOS-mod4=3-assum78}
\sum_{i=1}^N\langle B_iu, \widetilde{E}_mu\rangle^2=\sum_{i=1}^N\langle \widetilde{A}_iu, u\rangle^2=|u|^4, \quad u\in\mathbb{R}^l.
\end{equation}
Analogously, from (\ref{SOS-mod4=3-assum78}) we can derive formulae (\ref{SOS-mod4=3-assum2}-\ref{SOS-mod4=3-assum3}) for $\widetilde{A}_i$ in place of $A_i$. In particular,
the following equality holds
\begin{equation}\label{SOS-mod4=3-assum378}
\sum_{i=1}^N \Big(2(\widetilde{A}_i^r)^2+\mathrm{tr}(\widetilde{A}_i^1+\widetilde{A}_i^2)\widetilde{A}_i^r\Big)=18I_8, \quad r=1,2.
\end{equation}
In the same way, by restricting to $u=(u^1,0)\in \mathbb{R}^{8}  \oplus  \mathbb{R}^{8}$ or $u=(0, u^2)\in \mathbb{R}^{8}  \oplus  \mathbb{R}^{8}$, we can derive similar formulae (\ref{SOS-mod4=3-assum2}-\ref{SOS-mod4=3-assum3}) for $\widetilde{A}_i^r$ ($r=1,2$) in place of $A_i$ with $l$ replaced by $8$, i.e.,
\begin{equation}\label{SOS-mod4=3-assum378'}
\sum_{i=1}^N \Big(2(\widetilde{A}_i^r)^2+\mathrm{tr}(\widetilde{A}_i^r)\widetilde{A}_i^r\Big)=10I_8, \quad r=1,2.
\end{equation}
Since
$$0=\mathrm{tr}(A_i)=\mathrm{tr}(B_i^{11}E_m^1)+\mathrm{tr}(B_i^{22}E_m^1)=\mathrm{tr}(\widetilde{A}_i^1)-\mathrm{tr}(\widetilde{A}_i^2),$$
 it follows from (\ref{SOS-mod4=3-assum378}) and (\ref{SOS-mod4=3-assum378'}) that
\begin{equation}\label{SOS-mod4=3-assum378''}
\sum_{i=1}^N (\widetilde{A}_i^r)^2=-\sum_{i=1}^N(B_i^{rr})^2=I_8, \quad  \sum_{i=1}^N \mathrm{tr}(\widetilde{A}_i^r)\widetilde{A}_i^r=8I_8, \quad r=1,2.
\end{equation}
It follows from (\ref{SOS-mod4=3-assum378''}) and the Cauchy-Schwartz inequality
$$64=\sum_{i=1}^N \Big(\mathrm{tr}(\widetilde{A}_i^r)\Big)^2\leq \sum_{i=1}^N |B_i^{rr}|^2|E_m^1|^2=64, \quad r=1,2,$$
that there exist $b_i\in\mathbb{R}$ such that $\sum_{i=1}^N(b_i)^2=1$ and
\begin{equation}\label{Brr=Em}
B_i^{11}=b_iE_m^1=-B_i^{22}, \quad i=1,\cdots, N.
\end{equation}
Here the second equality holds because of the relation $\mathrm{tr}(B_i^{11}E_m^1)=-\mathrm{tr}(B_i^{22}E_m^1)$.
Combining (\ref{SOS-mod4=3-assum378''}) and (\ref{SOS-mod4=3-assum3}), we also have
$$\sum_{i=1}^NB_i^{12}(B_i^{12})^t=\sum_{i=1}^N(B_i^{12})^tB_i^{12}=8I_8,$$
which is, however, useless in deducing the contradiction.

Now we go back to analyze the equation (\ref{SOS-ind-assum2}) with respect to the irreducible decomposition (\ref{Clifforddecom-78}).
In this case, we rewrite (\ref{SOS-ind-assum2}) as:
\begin{equation}\label{SOS-78-assum2'}
\begin{array}{lll}
&\quad\sum\limits_{i=1}^N\Big(\langle B_i^{11}u^1,v^1\rangle+\langle B_i^{22}u^2,v^2\rangle+\langle B_i^{12} u^2, v^1\rangle-\langle B_i^{12} v^2, u^1\rangle\Big)^2\\
&=(|u^1|^2+|u^2|^2)(|v^1|^2+|v^2|^2)
-\Big(\langle u^1,v^1\rangle+\langle u^2,v^2\rangle\Big)^2 \\
&\quad -\sum\limits_{\alpha=1}^6\Big(\langle E_\alpha^1u^1,v^1\rangle+\langle E_\alpha^2u^2,v^2\rangle\Big)^2, \quad\quad\quad\quad\quad\quad\quad\quad u^1,u^2,v^1,v^2\in\mathbb{R}^8.
\end{array}
\end{equation}
In the same way as (\ref{Cliffordbasis}-\ref{SOS-ind-assum2-5}), by restricting to $u^2=v^2=0$ (or $u^1=v^1=0$) in (\ref{SOS-78-assum2'}) and using (\ref{Brr=Em}), firstly we see
$$\sum_{i=1}^N\langle B_i^{rr}u^1,v^1\rangle^2=\langle E_7^1 u^1, v^1\rangle^2=|u^1|^2|v^1|^2-\langle u^1,v^1\rangle^2-\sum_{\alpha=1}^6\langle E_\alpha^ru^1,v^1\rangle^2,$$
for any $u^1,v^1\in  \mathbb{R}^{8}$ and $r=1,2$. Next by restricting to $(v^1,v^2)=(-u^1,u^2)$ in (\ref{SOS-78-assum2'}), it follows
$$\sum_{i=1}^N\langle B_i^{12}u^2,u^1\rangle^2=|u^1|^2|u^2|^2, \quad (u^1,u^2)\in\mathbb{R}^8\oplus \mathbb{R}^8.$$
Substituting these identities into (\ref{SOS-78-assum2'}) and finally restricting to $(v^1,v^2)=(u^2,u^1)$ in (\ref{SOS-78-assum2'}) and using (\ref{Clifforddecom-irr-indef}), analogous to (\ref{SOS-ind-assum2-5}) we have the following
 \begin{equation}\label{SOS-ind-assum2-578}
 \begin{array}{ll}
\sum\limits_{i=1}^N\langle B_i^{12} u^1, u^1\rangle\langle B_i^{12} u^2, u^2\rangle=&
\langle u^1,u^2\rangle^2 + \langle E_7^1u^1,u^2\rangle^2-\sum\limits_{\alpha=1}^6\langle E_\alpha^1u^1,u^2\rangle^2\\
&+2\sum\limits_{i=1}^N\langle B_i^{11}u^1,u^2\rangle\Big(\langle B_i^{12} u^2, u^2\rangle-\langle B_i^{12} u^1, u^1\rangle\Big),
\end{array}
\end{equation}
where (\ref{Brr=Em}) has been used in the calculation
$$\sum_{i=1}^N\langle B_i^{11}u^1,u^2\rangle\langle B_i^{22}u^2,u^1\rangle=\langle E_7^1u^1,u^2\rangle^2.$$
Then as in the case of $(8,7)$ of indefinite class, for any fixed unit vector $u^1\in\mathbb{R}^8$,  $\{u^1, E_1^1u^1,\cdots,E_7^1u^1\}$ constitutes an orthonormal basis of $\mathbb{R}^8$.
Taking sum of (\ref{SOS-ind-assum2-578}) for $u^2=u^1, E_1^1u^1,\cdots,E_7^1u^1$, it gives
 \begin{equation}\label{SOS-ind-assum2-678}
\sum\limits_{i=1}^N\langle B_i^{12} u^1, u^1\rangle \mathrm{tr}(B_i^{12})=-4|u^1|^2,
\end{equation}
since $\langle B_i^{11}u^1,u^1\rangle=\langle B_i^{11}u^1,E_1^1u^1\rangle=\cdots=\langle B_i^{11}u^1,E_6^1u^1\rangle=0$ as $B_i^{11}$ and $B_i^{11}E_\alpha^1$ ($\alpha=1,\cdots,6$) are skew-symmetric, and $$\langle B_i^{12} E_7^1u^1,E_7^1u^1\rangle-\langle B_i^{12} u^1, u^1\rangle=0$$ as $ B_i^{12} E_7^1=E_7^1B_i^{12}$.
At last, taking contraction of (\ref{SOS-ind-assum2-678}) on $u^1$, we arrive at the following contradiction
\begin{equation*}
\sum\limits_{i=1}^N \Big(\mathrm{tr}(B_i^{12})\Big)^2=-32.  \qedhere
\end{equation*}
\end{proof}

\subsection{$m\equiv 1,~ 2 ~(mod ~4)$, $m\geq5$, $(m_+,m_-)\neq (5,2),~(6,1)$}
 Using the conclusions in the last subsection for the case of  $m\equiv 3~(mod ~4)$, we can show
\begin{thm}\label{NonSOS-mod4=12}
For the cases of $m\equiv 1,~ 2 ~(mod ~4)$, $m\geq5$, and $(m_+,m_-)\neq (5,2),~(6,1)$, the \emph{psd} form $G_F$ in $(\ref{nonnegativepolyG})$ is \emph{non-sos}.
\end{thm}
\begin{proof}
As in the last subsection, we prove it by contradiction. Assume there are quadratic forms $Q_1,\cdots,Q_N$ such that (\ref{SOS-ind-assum}) holds, i.e.,
\begin{equation*}
\sum_{i=1}^NQ_i(x)^2=G_F(x)=|x|^4-\sum_{\alpha=0}^m\langle P_\alpha x,x\rangle^2.
\end{equation*}
Firstly it follows from Table \ref{table3} that $l>8$ as $m\geq5$ and the cases of $(5,2),~(6,1)$ have been excluded.
For the Clifford system $\{P_0,\cdots,P_m\}$ on $\mathbb{R}^{2l}$, $\{P_0,\cdots,P_3\}$ is also a Clifford system on $\mathbb{R}^{2l}$ corresponding to an isoparametric polynomial
$F'(x)=|x|^4 - 2\sum_{\alpha = 0}^{3}\langle P_{\alpha}x,x\rangle^2$ of OT-FKM type with multiplicities $(3,l-4)$ (which is not equal to $(3,4)$).
Then the assumption above expresses the nonnegative polynomial $G_{F'}$ as a sum of squares of quadratic forms:
$$G_{F'}=|x|^4-\sum_{\alpha=0}^3\langle P_\alpha x,x\rangle^2=\sum_{i=1}^NQ_i(x)^2+\sum_{\alpha=4}^m\langle P_\alpha x,x\rangle^2,$$
which contradicts Theorem \ref{NonSOS-mod4=3}.
\end{proof}


\section{On isoparametric with $g=6$}\label{secg6}
In this section, we aim to prove that both \emph{psd} polynomials $G_F^{\pm}(x)$ in (\ref{psdform}) are not \emph{sos}, for isoparametric hypersurfaces with $g=6$ in $\mathbb{S}^7$ ($m_+=m_-=:m=1$) and in $\mathbb{S}^{13}$ ($m_+=m_-=:m=2$) respectively.

\begin{thm}\label{g6}
For $g=6$, both $G_F^{\pm}=|x|^6\pm F(x)\in P_{6m+2,6}$ in $(\ref{psdform})$ are \emph{non-sos}.
\end{thm}
In fact, for the case of $m=1$, we can establish the following stronger result.
\begin{thm}\label{g6m1}
Any cubic form on $\mathbb{R}^8$ vanishing on the focal submanifold $M_+$ or $M_-$ of dimension $5$ is identically zero. In particular, $M_+$ and $M_-$ are not cubic, i.e., not intersections of zeroes of cubic forms, and thus $G_F^{\pm}$ in $(\ref{psdform})$ are \emph{non-sos}.
\end{thm}
\begin{proof}
According to the classification of isoparametric hypersurfaces with $(g,m)=(6,1)$ by \cite{DN85}, the isoparametric polynomial $F(x)$ is uniquely determined up to congruences. A beautiful observation of Miyaoka \cite{Miy93} states that the isoparametric hypersurfaces are exactly the pull-back of the isoparametric hypersurfaces with $(g,m)=(3,1)$ through the Hopf fiberation. Then we can write the isoparametric polynomial of degree $6$ as the composition $F=F_C\circ\pi$, where $F_C$ is Cartan's isoparamtric polynomial of degree $3$ as in (\ref{Cartanpoly}), and $\pi: \mathbb{R}^8 \rightarrow \mathbb{R}^5$ is the Hopf fiberation
$$\pi(u,v)=(|v|^2-|u|^2,2u\bar{v}), \quad \quad \quad x=(u,v) \in  \mathbb{H}\oplus\mathbb{H}\cong\mathbb{R}^8. $$

Let $\mathcal{V}: \mathbb{S}^2\times\mathbb{S}^3\rightarrow \mathbb{S}^7\subset  \mathbb{H}\oplus\mathbb{H}\cong\mathbb{R}^8$ be the map
\begin{equation}\label{VM+}
\mathcal{V}(t,q)=\Big(\frac{\sqrt{3}}{2}(t_1\oi+t_2\oj)q,~(t_0+\frac{1}{2}t_1\oi-\frac{1}{2}t_2\oj)q\Big), \quad   t=(t_0,t_1,t_2)\in \mathbb{S}^2,~q\in \mathbb{S}^3.
 \end{equation}
  We claim that the image $\mathcal{V}(\mathbb{S}^2\times\mathbb{S}^3)$ is exactly the focal submanifold $M_+=F^{-1}(1)\cap\mathbb{S}^7=\pi^{-1}(M_+^C)$ (diffeomorphic to $ \mathbb{R}P^2\times\mathbb{S}^3$), where $M_+^C:=F_C^{-1}(1)\cap\mathbb{S}^4$ is the focal submanifold of $F_C$, i.e., the Veronese surface $\mathbb{R}P^2$ in $\mathbb{S}^4$. In fact, the Cartan polynomial on $\mathbb{R}^5$ of degree $3$ in (\ref{Cartanpoly}) can be rewritten as
$$F_C(y)=\frac{3\sqrt{3}}{2}\det\begin{pmatrix}
-\frac{1}{\sqrt{3}}y_0+y_1 & y_2 & y_4 \\
y_2 & \frac{2}{\sqrt{3}}y_0 & y_3 \\
y_4 & y_3 & -\frac{1}{\sqrt{3}}y_0-y_1
\end{pmatrix}, \quad  y=(y_0,\cdots,y_4)\in\mathbb{R}^5.$$
Then let $y$ be a point in the image of $\pi\circ \mathcal{V}$, i.e.,
\begin{equation}\label{M+c}
y=\pi\circ \mathcal{V}(t,q)=\Big(\frac{1}{2}(2t_0^2-t_1^2-t_2^2), \frac{\sqrt{3}}{2}(t_1^2-t_2^2), \sqrt{3}t_0t_1,  \sqrt{3}t_0t_2,  \sqrt{3}t_1t_2\Big).
\end{equation}
It follows that
$$F\circ\mathcal{V}(t,q)=F_C(y)=\frac{3\sqrt{3}}{2}\det\Big(-\frac{1}{\sqrt{3}}I_3+\sqrt{3}\begin{pmatrix}t_1\\t_0\\t_2\end{pmatrix}\begin{pmatrix}t_1 &t_0& t_2\end{pmatrix}\Big)=1.$$

Now let $\Phi(x)$ be a cubic form on $\mathbb{R}^8$ vanishing on the focal submanifold $M_+$. Decompose it as
$$\Phi(u,v)=\psi(v)+P(u,v)+Q(u,v)+\varphi(u), \quad \quad \quad x=(u,v) \in  \mathbb{H}\oplus\mathbb{H}\cong\mathbb{R}^8,$$
where $\psi, P, Q, \varphi$ are cubic forms with degree $0,1,2,3$ on $u\in\mathbb{H}\cong\mathbb{R}^4$ (and thus with degree $3,2,1,0$ on $v\in\mathbb{H}$) respectively. For example, we can set
$P(u,v)=P_u(v)$ where $P_u:=\sum_{i=1}^4u_iP_i$ for certain real symmetric $(4\times 4)$ matrices $P_i$'s, $u=(u_1,\cdots,u_4)\in\mathbb{H}\cong\mathbb{R}^4$, and $P_u(v)=P_u(v,v)=\langle P_uv, v\rangle$ is the quadratic form associated to $P_u$.

As $M_+$ is parameterized by $(u,v)=\mathcal{V}(t,q)$ in (\ref{VM+}), we investigate firstly the evaluation of $\Phi$ on the points with $t_1=t_2=0$.
It follows that $\Phi(0,v)=\psi(v)\equiv 0$ for any $v \in  \mathbb{H}$, and thus $$\Phi(u,v)=P(u,v)+Q(u,v)+\varphi(u).$$
Then we consider the evaluation of $\Phi$ on the points with $r:=\sqrt{t_1^2+t_2^2}>0$.
Setting $t_1/r=:\cos\theta$, $t_2/r=:\sin\theta$ and $w:=(\cos\theta\oi+\sin\theta\oj)q\in\mathbb{S}^3$, we calculate
\begin{equation}\label{M+cp}
\begin{aligned}
&\Phi\Big(\frac{\sqrt{3}}{2}(t_1\oi+t_2\oj)q,~(t_0+\frac{1}{2}t_1\oi-\frac{1}{2}t_2\oj)q\Big)\\
&=\Phi\Big(\frac{\sqrt{3}}{2}rw, -t_0(\cos\theta\oi+\sin\theta\oj)w+\frac{1}{2}re^{-2\theta\ok}w\Big)\\
&=\frac{\sqrt{3}}{2}rP\Big(w,-t_0(\cos\theta\oi+\sin\theta\oj)w+\frac{1}{2}re^{-2\theta\ok}w\Big)\\
&\quad+\frac{3}{4}r^2\left(-t_0Q\Big(w,(\cos\theta\oi+\sin\theta\oj)w\Big)+\frac{1}{2}rQ\Big(w,e^{-2\theta\ok}w\Big)\right)\\
&\quad+\frac{3\sqrt{3}}{8}r^3\varphi(w)\\
&=\frac{\sqrt{3}}{2}rt_0^2P\Big(w,(\cos\theta\oi+\sin\theta\oj)w\Big)\\
&\quad-\frac{\sqrt{3}}{2}r^2t_0\left(P_w\Big((\cos\theta\oi+\sin\theta\oj)w, e^{-2\theta\ok}w\Big)+\frac{\sqrt{3}}{2}Q\Big(w,(\cos\theta\oi+\sin\theta\oj)w\Big)\right)\\
&\quad+\frac{\sqrt{3}}{8}r^3\left(P\Big(w,e^{-2\theta\ok}w\Big)+\sqrt{3}Q\Big(w,e^{-2\theta\ok}w\Big)+3\varphi(w)\right)\\
&\equiv0,
\end{aligned}
\end{equation}
for any $0<r<1$ with $r^2+t_0^2=1$, $\theta\in\mathbb{R}$ and for any $w\in\mathbb{H}$. By comparing the degree of $r$ or using a coordinate translation $(r=\cos\phi, t_0=\sin\phi)$, one can easily deduce from the preceding identity that
\begin{eqnarray}
&P\Big(w,(\cos\theta\oi+\sin\theta\oj)w\Big)\equiv0, \label{t2}\\
&P_w\Big((\cos\theta\oi+\sin\theta\oj)w, e^{-2\theta\ok}w\Big)+\frac{\sqrt{3}}{2}Q\Big(w,(\cos\theta\oi+\sin\theta\oj)w\Big)\equiv0,\label{t1}\\
&P\Big(w,e^{-2\theta\ok}w\Big)+\sqrt{3}Q\Big(w,e^{-2\theta\ok}w\Big)+3\varphi(w)\equiv0, \quad \textit{for any }\theta\in\mathbb{R}, w\in\mathbb{H}.\label{t0}
\end{eqnarray}

The identity (\ref{t2}) will lead to
$$\cos^2\theta~ P(w,\oi w)+\sin^2\theta~ P(w,\oj w)+\sin2\theta~ P_w(\oi w, \oj w)\equiv0, \quad \textit{for any }\theta\in\mathbb{R}, w\in\mathbb{H}.$$
This implies
\begin{equation}\label{eqt2}
P(w,\oi w)=P(w,\oj w)=P_w(\oi w, \oj w)\equiv0, \quad\quad \textit{for any }  w\in\mathbb{H}.
\end{equation}

Computing the identity (\ref{t1}), we obtain
$$\begin{aligned}
&\cos\theta\cos2\theta ~P_w(\oi w, w)-\sin\theta\sin2\theta~ P_w(\oj w, \ok w)+\frac{\sqrt{3}}{2}\cos\theta~ Q(w,\oi w)\\
&+\sin\theta\cos2\theta~ P_w(\oj w, w)-\cos\theta\sin2\theta~ P_w(\oi w, \ok w) +\frac{\sqrt{3}}{2}\sin\theta ~Q(w,\oj w)\equiv0,
\end{aligned}$$
for any $\theta\in\mathbb{R}$ and $w\in\mathbb{H}$. This implies
$$\begin{aligned}
&\cos\theta\cos2\theta ~P_w(\oi w, w)-\sin\theta\sin2\theta~ P_w(\oj w, \ok w)+\frac{\sqrt{3}}{2}\cos\theta~ Q(w,\oi w)\equiv0,\\
&\sin\theta\cos2\theta~ P_w(\oj w, w)-\cos\theta\sin2\theta~ P_w(\oi w, \ok w) +\frac{\sqrt{3}}{2}\sin\theta ~Q(w,\oj w)\equiv0,
\end{aligned}$$
and thus
\begin{equation}\label{eqt1}
\begin{aligned}
&-P_w(\oi w, w)=P_w(\oj w, \ok w)=\frac{\sqrt{3}}{2}Q(w,\oi w)=:\beta(w)=:\beta,\\
&P_w(\oj w, w)=P_w(\oi w, \ok w) =\frac{\sqrt{3}}{2}Q(w,\oj w)=:\gamma(w)=:\gamma, \quad \textit{for any }  w\in\mathbb{H},
\end{aligned}
\end{equation}
where $\beta,\gamma$ are denoted to be the corresponding cubic forms.

Computing the identity (\ref{t0}), we deduce
$$\begin{aligned}
&\cos^22\theta~P(w,w)+\sin^22\theta~P(w,\ok w)+3\varphi(w)\\
&-\sin4\theta~P_w(w,\ok w)+\sqrt{3}\cos2\theta~Q(w,w)-\sqrt{3}\sin2\theta~Q(w,\ok w)\equiv0,
\end{aligned}$$
for any $\theta\in\mathbb{R}$ and $w\in\mathbb{H}$. This implies
\begin{equation}\label{eqt0}
\begin{aligned}
&P(w,w)=P(w,\ok w)=-3\varphi(w)=:\alpha(w)=:\alpha,\\
&P_w(w,\ok w)=Q(w,w)=Q(w,\ok w)\equiv0,\quad \textit{for any }  w\in\mathbb{H}£¬
\end{aligned}
\end{equation}
where $\alpha$ is denoted to be the corresponding cubic form.

As $\{w, \oi w, \oj w, \ok w\}$ form an orthonormal basis of $\mathbb{H}$ for any $w\in\mathbb{S}^3$, under this basis we deduce the matrix $P_w$ from (\ref{eqt2}, \ref{eqt1}, \ref{eqt0}) as
$$|w|^2P_w (w, \oi w, \oj w, \ok w)= (w, \oi w, \oj w, \ok w)
\begin{pmatrix}
\alpha & -\beta&\gamma&0\\
-\beta&0&0&\gamma\\
\gamma&0&0&\beta\\
0&\gamma&\beta&\alpha
\end{pmatrix}.$$
Or alternatively,
\begin{equation}\label{Pweq}
|w|^4P_w=(w, \oi w, \oj w, \ok w)
\begin{pmatrix}
\alpha & -\beta&\gamma&0\\
-\beta&0&0&\gamma\\
\gamma&0&0&\beta\\
0&\gamma&\beta&\alpha
\end{pmatrix}(w, \oi w, \oj w, \ok w)^t.
\end{equation}
Let $w=\mathbf{1},~\oi,~\oj,~\ok$ respectively, and let $\alpha_i,\beta_i,\gamma_i, P_i$ ($i=1,2,3,4$) denote the corresponding values of $\alpha(w),\beta(w),\gamma(w)$ and $P_w$.
In this way,
\begin{equation}\label{Pi}
\begin{aligned}
&P_1=\begin{pmatrix}
\alpha_1 & -\beta_1&\gamma_1&0\\
-\beta_1&0&0&\gamma_1\\
\gamma_1&0&0&\beta_1\\
0&\gamma_1&\beta_1&\alpha_1
\end{pmatrix},\quad P_2=\begin{pmatrix}
0 & \beta_2&-\gamma_2&0\\
\beta_2&\alpha_2&0&-\gamma_2\\
-\gamma_2&0&\alpha_2&-\beta_2\\
0&-\gamma_2&-\beta_2&0
\end{pmatrix},\\
&P_3=\begin{pmatrix}
0 & \beta_3&-\gamma_3&0\\
\beta_3&\alpha_3&0&-\gamma_3\\
-\gamma_3&0&\alpha_3&-\beta_3\\
0&-\gamma_3&-\beta_3&0
\end{pmatrix},\quad  P_4=\begin{pmatrix}
\alpha_4 & -\beta_4&\gamma_4&0\\
-\beta_4&0&0&\gamma_4\\
\gamma_4&0&0&\beta_4\\
0&\gamma_4&\beta_4&\alpha_4
\end{pmatrix}.
\end{aligned}
\end{equation}
Noting that $P_w=\sum_{i=1}^4w_iP_i$ is linear about $w$, we see that the $(1,4)$ entry and $(2,3)$ entry of $P_w$ vanish for any $w\in\mathbb{H}$.
Calculating the $(1,4)$ entry and $(2,3)$ entry of $|w|^4P_w$ from the preceding formula (\ref{Pweq}), we obtain the following two equations:
$$\begin{aligned}
&-2\gamma(w_1w_2+w_3w_4)-2\beta(w_1w_3-w_2w_4)\equiv0,\\
&2\gamma(w_1w_2+w_3w_4)-2\beta(w_1w_3-w_2w_4)\equiv0,
\end{aligned}$$
which implies that $\beta(w)=\gamma(w)\equiv0$. Now since $Q(u,v)$ is linear about $v$, we conclude from (\ref{eqt1}, \ref{eqt0}) that $Q(u,v)\equiv0$.
Moreover, $P_w$ can be given explicitly
$$\begin{aligned}
|w|^4P_w
&=(w, \oi w, \oj w, \ok w)
\begin{pmatrix}
\alpha & 0&0&0\\
0&0&0&0\\
0&0&0&0\\
0&0&0&\alpha
\end{pmatrix}(w, \oi w, \oj w, \ok w)^t\\
&=\alpha\begin{pmatrix}
 w_1^2+w_4^2& w_1w_2+w_3w_4& w_1w_3-w_2w_4&0\\
 w_1w_2+w_3w_4&w_2^2+w_3^2&0&w_2w_4-w_1w_3\\
w_1w_3-w_2w_4&0&w_2^2+w_3^2&w_1w_2+w_3w_4\\
0&w_2w_4-w_1w_3&w_1w_2+w_3w_4&w_1^2+w_4^2
\end{pmatrix}.
\end{aligned}$$

On the other hand, a calculation by using formula (\ref{Pi}) implies
$$P_w=\sum_{i=1}^4w_iP_i=\odiag(\alpha_1w_1+\alpha_4w_4, \alpha_2w_2+\alpha_3w_3, \alpha_2w_2+\alpha_3w_3, \alpha_1w_1+\alpha_4w_4).$$
Combining these two expressions of $P_w$ will lead us to
$$\begin{aligned}
&|w|^4(\alpha_1w_1+\alpha_4w_4)\equiv\alpha( w_1^2+w_4^2), \quad |w|^4(\alpha_2w_2+\alpha_3w_3)=\alpha(w_2^2+w_3^2),\\
&\alpha(w_1w_2+w_3w_4)=\alpha(w_1w_3-w_2w_4)\equiv0, \quad \quad \textit{for any }  w\in\mathbb{H}.
\end{aligned}$$
This implies that $\alpha\equiv0$, $P_w\equiv0$ and thus $P(u,v)=P_u(v)\equiv0$, and $\varphi(u)\equiv0$ by (\ref{eqt0}).
In conclusion, $\Phi(u,v)=P(u,v)+Q(u,v)+\varphi(u)\equiv0$, any cubic form $\Phi$ vanishing on $M_+$ is identically zero as desired.

Now we turn to consider the question on $M_-:=F^{-1}(-1)\cap\mathbb{S}^7$ which is diffeomorphic but not antipodal (nor isometric) to $M_+$. A cubic form vanishing on $M_-$ may not vanish on $M_+$. However, the images of $M_{\pm}$ under the Hopf fiberation $\pi$ $$M_-^C:=\pi(M_-)=F_C^{-1}(-1)\cap \mathbb{S}^4=-F_C^{-1}(1)\cap \mathbb{S}^4=-\pi(M_+)=-M_+^C$$  are antipodal to each other. Observing the identity (\ref{M+cp}), we parameterize points of $M_+$ alternatively by
$$M_+=\Big\{\Big(\frac{\sqrt{3}}{2}\cos\phi~ w, -\sin\phi~(\cos\theta\oi+\sin\theta\oj)w+\frac{1}{2}\cos\phi~ e^{-2\theta\ok}w\Big) \mid \theta,\phi\in\mathbb{R},~w\in\mathbb{S}^3\Big\}.$$
A long but straightforward calculation shows that $M_-$ can also be parameterized as $M_+$ above by
$$M_-=\Big\{\Big(\sin\phi~(\cos\theta\oi+\sin\theta\oj)w+\frac{1}{2}\cos\phi~ e^{2\theta\ok}w, -\frac{\sqrt{3}}{2}\cos\phi~ w\Big) \mid \theta,\phi\in\mathbb{R},~w\in\mathbb{S}^3\Big\}.$$
In fact, let $t_0:=\sin\phi,~ t_1:=\cos\phi\cos\theta,~ t_2:=\cos\phi\sin\theta$ as before, then it can be easily verified from $M_-=\pi^{-1}(-M_+^C)$ by the parametrization of $M_+^C$ in (\ref{M+c}). Then by the same argument as on $M_+$, we can show that any cubic form $\Phi$ vanishing on $M_-$ is identically zero.
\end{proof}

To prove Theorem \ref{g6}, the last case we are left with considering is the isoparametric polynomial with $(g,m)=(6,2)$. Fortunately, isoparametric hypersurfaces in this case have been classified by Miyaoka (\cite{Miy13, Miy16}) to be the unique homogeneous class. Moreover, the explicit formulae of the isoparametric polynomials representing the homogeneous isoparametric hypersurfaces with $g=6$, $m=1,2$, denoted by $F_1(x)$ ($x\in\mathbb{R}^8$) and $F_2(X)$ ($X\in\mathbb{R}^{14}$) respectively, were given by Ozeki and Takeuchi \cite{OT75}, and then were simplified by Peng and Hou \cite{PH89}. The points $X\in\mathbb{R}^{14}$ are written in terms of the skew-Hermitian matrix representation of the exceptional simple Lie algebra $\mathfrak{g}_2$, while the points $x\in\mathbb{R}^8$ are identified with the real symmetric matrices in $\mathfrak{p}$ of the Cartan decomposition $\mathfrak{g}_2=\mathfrak{k}+\sqrt{-1}\mathfrak{p}$ with $\mathfrak{k}$ being the real skew-symmetric part. Explicitly, we can write $X=K+\sqrt{-1}x$ with the real symmetric part $x$ and the real skew-symmetric part $K$ in the following form (\cite{PH89}):
\begin{equation}\label{g6m2}
x=\begin{pmatrix}
0&Y&-Y\\
Y^t&T&S\\
-Y^t&-S&-T
\end{pmatrix}, \quad K=\begin{pmatrix}
0&u&u\\
-u^t&U&V\\
-u^t&V&U
\end{pmatrix},
\end{equation}
where  $$Y=\frac{1}{\sqrt{3}}(y_1,y_2,y_3), \quad  u=\frac{1}{\sqrt{3}}(u_1,u_2,u_3),$$
$$\begin{aligned}
&T=\begin{pmatrix}
t_1&\frac{1}{\sqrt{2}}y_4&\frac{1}{\sqrt{2}}y_5\\
\frac{1}{\sqrt{2}}y_4&t_2&\frac{1}{\sqrt{2}}y_6\\
\frac{1}{\sqrt{2}}y_5&\frac{1}{\sqrt{2}}y_6&t_3
\end{pmatrix}, \quad
S=\frac{1}{\sqrt{6}}\begin{pmatrix}
0&-y_3&y_2\\
y_3&0&-y_1\\
-y_2&y_1&0
\end{pmatrix};\\
&U=\frac{1}{\sqrt{2}}\begin{pmatrix}
0&u_4&u_5\\
-u_4&0&u_6\\
-u_5&-u_6&0
\end{pmatrix}, \quad
V=\frac{1}{\sqrt{6}}\begin{pmatrix}
0&-u_3&u_2\\
u_3&0&-u_1\\
-u_2&u_1&0
\end{pmatrix};
\end{aligned}$$
where $t_1+t_2+t_3=0$, $y_i,u_i$ ($i=1,\cdots,6$) are real numbers. Then the isoparametric polynomials with $g=6$, $m=1,2$, can be given as
\begin{equation}\label{isopg6}
F_1(x):=18 \mathrm{tr}(x^6) -\frac{5}{4}(\mathrm{tr}(x^2))^3,\quad  F_2(X):=18 \mathrm{tr}(X^6) -\frac{5}{4}(\mathrm{tr}(X^2))^3.
\end{equation}
Clearly we have $F_1(x)=-F_2(X)$ for $X=\sqrt{-1}x$ with $K=0$.

Now we are ready to prove Theorem \ref{g6} for the last case $(g,m)=(6,2)$. Assume that $G_{F_2}^{\pm}(X):=|X|^6\pm F_2(X)=\sum_{\alpha}\Phi_{\alpha}(X)^2$ is a sum of squares of some cubic forms $\Phi_{\alpha}(X):=\Phi_{\alpha}(K,x)$ on $X=K+\sqrt{-1}x:=(K,x)\in\mathbb{R}^{14}$.
Then restricting to $\sqrt{-1}\mathfrak{p}$ ($X=\sqrt{-1}x=(0,x)\in\mathbb{R}^8$ with $K=0$), we deduce
$$G_{F_1}^{\mp}(x):=|x|^6\mp F_1(x)=|X|^6\pm F_2(X)=\sum_{\alpha}\Phi_{\alpha}(0,x)^2$$
is a sum of squares of cubic forms. This contradicts Theorem \ref{g6m1}. \qed  

We conclude this section with a remark. It can be conjectured that a statement similar to Theorem \ref{g6m1} holds in the case of $(g,m)=(6,2)$. More precisely,
we conjecture that any cubic form on $\mathbb{R}^{14}$ vanishing on the focal submanifold $M_+$ or $M_-$ of dimension $10$ is identically zero.
Unfortunately, due to extremely complicated computations, we failed to give a proof.

\section{Further remarks and applications}\label{sec-6}
In this section we present further discussions on the \emph{psd} forms $G_F^+$ of (\ref{psdform}) and their zeroes.

\subsection{Relations with Lagrange's identity}
Let us start with the \emph{psd} forms $G_F^+$, denoted by $G_{km}(x):=\frac{1}{8}G_F^+(x)$, for isoparametric polynomials on $\mathbb{R}^{2l}$ of OT-FKM type with $g=4$, $l=k\delta(m)$, $k\geq2$, $m=1,2,3$ and the definite class of $m=4$, respectively.

The nonnegativity of $G_{km}(x)$ shows Cauchy-Schwarz's inequality for real, complex and quaternionic vectors in $\mathbb{F}^k$ with $\mathbb{F}=\mathbb{R}, \mathbb{C}, \mathbb{H}$ for $m=1,2,4$, respectively, namely,
\begin{equation}\label{CS-non-sos1}
G_{km}(X,Y)=|X|^2|Y|^2-|\langle X,Y\rangle_{H}|^2\geq0, \quad \textit{for}~x=(X,Y)\in \mathbb{F}^k\oplus\mathbb{F}^k,~k\geq2,
\end{equation}
where $\langle X,Y\rangle_{H}=\sum_{i=1}^kX_i\overline{Y_i}$ is the Hermitian product. For $m=1,2$, the nonnegativity also follows from the Lagrange identity
$$|X|^2|Y|^2-|\langle X,Y\rangle_{H}|^2=|X\wedge Y|^2, \quad \textit{for}~X,Y\in \mathbb{R}^k~\textit{or}~\mathbb{C}^k,~k\geq2,$$
 which is exactly Solomon's \emph{sos} expression of $G_{km}(x)$ in (\ref{SOS-FKM1}). However for $m=4$, usually one needs another approach to prove the nonnegativity (see \cite{GLZ18} for example), since there no longer holds Lagrange's identity for quaternions; in fact, there do not even exist \emph{sos} expressions for $G_{km}(x)$ (proved in Theorem \ref{defcase}). As for the expression in (\ref{CS-non-sos1}), we recall (\ref{SOS-ind-assum2}) where $G_{km}(x)=\frac{1}{8}G_F^+(x)\in P_{2l,4}\setminus \Sigma_{2l,4}$ can be rewritten as
 \begin{equation}\label{identity-Gqm}
G_{km}(u,v)=|u|^2|v|^2-\langle u,v\rangle^2-\sum_{\alpha=1}^{m-1}\langle E_\alpha u,v\rangle^2, \quad x=(u,v)\in\mathbb{R}^l\oplus\mathbb{R}^l.
\end{equation}
 Now $m=4$, $l=4k$, there is a natural isomorphism $\mathbb{R}^l\cong\mathbb{H}^k$ which identifies $(u,v)=(X,Y)\in\mathbb{H}^k\oplus\mathbb{H}^k$. Because the Clifford system is definite,  the Clifford algebra $\{E_1,E_2,E_3\}$ on $\mathbb{R}^{4k}$ corresponds to the left quaternionic product by $\{\oi,\oj,\ok\}$ on $\mathbb{H}^k$ and thus
$$|\langle X,Y\rangle_{H}|^2=\langle u,v\rangle^2+\langle E_1u,v\rangle^2+\langle E_2u,v\rangle^2+\langle E_3u,v\rangle^2.$$

On the other hand, Proposition \ref{prop-sosrx24} tells us the following \emph{sosr} expression for $G_{k4}(x)$ with a uniform denominator $|x|^2$:
$$4|x|^2(|X|^2|Y|^2-|\langle X,Y\rangle_{H}|^2)=|\nabla G_{k4}(x)|^2,\quad \textit{for}~x=(X,Y)\in \mathbb{H}^k\oplus\mathbb{H}^k,~k\geq2.$$
However, this does not generalize Lagrange's identity for quaternions. Alternatively, we use the  \emph{sos} expression for $H_F(x)$ in Proposition \ref{prop-HF}:
 $$|\nabla F(x)\wedge x|^2=16H_F(x)=16G_F^+(x)G_F^-(x).$$
This indeed generalizes Lagrange's identity for quaternions:
\begin{equation}\label{sosr-Lag}
4\Big(\sum_{\alpha=0}^4\langle P_{\alpha}x,x\rangle^2\Big)\Big(|X|^2|Y|^2-|\langle X,Y\rangle_{H}|^2\Big)=\Big|\sum_{\alpha=0}^4\langle P_{\alpha}x,x\rangle P_\alpha x ~\wedge~x\Big|^2,
\end{equation}
where for $x=(X,Y)$,
\begin{equation}\label{cliffordm4}
\begin{aligned}
&P_0x=(X,-Y), \quad P_1x=(Y,X),\\
&P_{\alpha+1} x=(E_\alpha Y, -E_\alpha X)=(\oi Y, -\oi X),~ (\oj Y, -\oj X),~ (\ok Y, -\ok X),
\end{aligned}
\end{equation}
 for $\alpha=1,2,3$, respectively. Note also that when $X,Y\in\mathbb{R}^k$ or $\mathbb{C}^k$, (\ref{sosr-Lag}) reduces to the classical Lagrange identity.
Summarizing the arguments above, we have shown
\begin{prop}\label{propA}
For $~X,Y\in \mathbb{H}^k,~k\geq 2$, the \emph{psd} polynomial
$|X|^2|Y|^2-|\langle X,Y\rangle_{H}|^2 $ is not \emph{sos.} However, the polynomial
$(|X|^2+|Y|^2)(|X|^2|Y|^2-|\langle X,Y\rangle_{H}|^2 )$
is \emph{ sos} with a concrete representation. Furthermore, a generalized Lagrange identity $(\ref{sosr-Lag})$ holds for the quaternionic case.
\end{prop}

Now we turn to the case of $m=3$. The Clifford system $\{P_0,\cdots,P_3\}$ is just that of $m=4$ by deleting $P_4$ from (\ref{cliffordm4}). In this way, we can rewrite the \emph{psd} form $G_{k3}(x)\in P_{8k,4}$, $k\geq2$, by: for $x=(X,Y)\in \mathbb{H}^k\oplus\mathbb{H}^k$,
\begin{equation}\label{CS-non-sos2}
\begin{aligned}
G_{k3}(X,Y)&=G_{k4}(X,Y)+\Big(\oRe (\ok\langle X,Y\rangle_{H})\Big)^2\\
&=|X|^2|Y|^2-|\langle X,Y\rangle_{H}|^2+\Big(\oRe (\ok\langle X,Y\rangle_{H})\Big)^2.
\end{aligned}
\end{equation}
Similarly, we can rewrite the cases of $m=1,2$ for quaternions by:
\begin{equation}\label{CS-sos3}
\begin{aligned}
G_{2k,2}(X,Y)&=G_{k3}(X,Y)+\Big(\oRe (\oj\langle X,Y\rangle_{H})\Big)^2\\
&=|X|^2|Y|^2-\Big(\oRe (\langle X,Y\rangle_{H})\Big)^2-\Big(\oRe (\oi\langle X,Y\rangle_{H})\Big)^2;
\end{aligned}
\end{equation}
\begin{equation}\label{CS-sos4}
\begin{aligned}
G_{4k,1}(X,Y)&=G_{2k,2}(X,Y)+\Big(\oRe (\oi\langle X,Y\rangle_{H})\Big)^2\\
&=|X|^2|Y|^2-\Big(\oRe (\langle X,Y\rangle_{H})\Big)^2,
\end{aligned}
\end{equation}
for $X,Y\in \mathbb{H}^k\cong\mathbb{C}^{2k}\cong\mathbb{R}^{4k}$, $k\geq2$.

Noting that $G_{k3}$ corresponds to $\frac{1}{8}G_F^+$ for isoparametric polynomials of OT-FKM type with $(m_+,m_-)=(3,4k-4)$, we conclude from Theorem \ref{SOS-FKM} the following
\begin{cor}
 The \emph{psd} form $G_{k3}\in P_{8k,4}$ is \emph{sos} if $k=2$, and \emph{non-sos} if $k\geq3$.
\end{cor}
\begin{rem}
From the corollary above and \emph{(\ref{CS-non-sos2}-\ref{CS-sos4})}, one can see that the \emph{non-sos} \emph{psd} form  $G_{k4}\in P_{8k,4}$ $($resp. $G_{k3}\in P_{8k,4}$$)$ would turn to be \emph{sos} if an additional square of a quadratic form is added in the case of $k=2$ $($resp. of $k\geq3$$)$.
\end{rem}

\subsection{Applications to orthogonal multiplication}

Recall that an orthogonal multiplication of type $[p,q,r]$, $p\leq q$, is a bilinear map
$$T:\mathbb{R}^p\times\mathbb{R}^q\rightarrow\mathbb{R}^r$$
 such that $|T(u,v)|=|u||v|$ for all $u\in \mathbb{R}^q$ and $v\in \mathbb{R}^q$. The existence problem of a given type $[p,q,r]$ has been studied many times but is still open (see \cite{Lam85} and references therein). The case when $p=q$ is of particular interest for its important applications in geometry, e.g., harmonic maps from $\mathbb{S}^{2p-1}$ to $\mathbb{S}^r$ by Hopf construction. For example,
 this will produce the classical harmonic maps (i.e., Hopf fibrations) from $\mathbb{S}^{2m-1}$ to $\mathbb{S}^{m}$ with $m=1, 2, 4, 8$.
 Furthermore, by
 deforming the Hopf construction into a harmonic map, one can establish many harmonic representations
  in the classes of homotopy groups of spheres (see \cite{PT97, PT98}).

Now for the infinitely many classes of isoparametric polynomials $F(x)$ of OT-FKM type, we have shown in (\ref{SOS-ind-assum2}) that
\begin{equation}\label{Rewritten}
\frac{1}{8}G_F^+(x)=|u|^2|v|^2-\langle u,v\rangle^2-\sum_{\alpha=1}^{m-1}\langle E_\alpha u,v\rangle^2\in P_{2l,4}, \quad x=(u,v)\in\mathbb{R}^l\times\mathbb{R}^l.
\end{equation}
This result shows immediately:
\begin{cor}
Orthogonal multiplications of type $[l,l,m+r]$ $(r\geq1)$ with the form
$$T(u,v)=\Big(\langle u,v\rangle, \langle E_1 u,v\rangle, \cdots, \langle E_{m-1} u,v\rangle, T_{1}(u,v),\cdots,T_{r}(u,v) \Big),$$
where $T_i(u,v)$'s are nonzero bilinear functions, exist if and only if the Clifford algebra $\{E_1,\cdots,E_{m-1}\}$ on $\mathbb{R}^{2l}$ occurs in the \emph{sos} cases of Theorem $\ref{SOS-FKM}$. In this case,
$$(m,l)=(1,k)\footnote{Rigorously, the case of $(m,l)=(1,2)$ does not come from isoparametric polynomials with $g=4$ because now $m_-=l-m-1=0$.}, ~(2,2k),~ (3,8), ~(4,8),~ (5,8),~ (6,8),\quad k\geq2.$$
\end{cor}

\subsection{Applications to the \emph{sos} problem on Grassmannian $Gr_2(\mathbb{R}^l)$}
The \emph{non-sos} \emph{psd} forms $G_F^+\in P_{2l,4}$ will provide examples of \emph{psd} quadratic forms that are \emph{non-sos} on the oriented Grassmannian $Gr_2(\mathbb{R}^l)$. For our purpose, we regard $Gr_p(\mathbb{R}^l)$ as a quadratic variety in $\mathbb{S}^{\binom{l}{p}-1}\subset \Lambda^p(\mathbb{R}^l)$ by the Pl\"{u}cker relations, though $Gr_2(\mathbb{R}^l)\cong Q^{l-2}(\mathbb{C})=\{[z]\in\mathbb{CP}^{l-1}\mid z_1^2+\cdots+z_l^2=0\}$ is more understandable, via $x\wedge y\rightarrow x+\oi y$ where
$x, y \in\mathbb{R}^l$ is an orthonormal basis of an oriented plane.

Recall that a real homogeneous polynomial is called nonnegative (\emph{psd}) on a variety $X\subset\mathbb{CP}^n$ if its evaluation at each real point is nonnegative, and is called \emph{sos} on $X$ if it is a polynomial sum of squares modulo the defining polynomials, ideal $I(X)$ consisting of polynomials vanishing on $X$. For example, if $X$ is the Grassmannian, the ideal $I(X)$ is generated by the quadratic forms of the Pl\"{u}cker relations (cf. \cite{Ha95}). In particular, an exterior $2$-form $\omega\in\Lambda^2(\mathbb{F}^l)$ ($\mathbb{F}=\mathbb{R}$ or $\mathbb{C}$) is decomposable (namely, belongs to $Gr_2(\mathbb{F}^l)$) if and only if $\omega\wedge\omega=0$, which represents $\binom{l}{4}$ independent quadratic relations.

A celebrated result of Blekherman-Smith-Velasco \cite{BSV16} identifies all of the real projective varieties $X\subset\mathbb{CP}^n$, on which every nonnegative quadratic function is a polynomial sum of squares modulo the defining ideal of $X$, to be those varieties of minimal degree $\mathrm{deg}(X)=\mathrm{codim}(X)+1$. This generalizes Hilbert's theorem from $X=\mathbb{CP}^1\subset\mathbb{CP}^{\frac{d}{2}}$ by Veronese maps (i.e., $P_{2,d}=\Sigma_{2,d}$), or $X=\{[x^2:xy:xz:y^2:yz:z^2]\in\mathbb{CP}^5\mid [x:y:z]\in \mathbb{CP}^2\}$ (i.e., $P_{3,4}=\Sigma_{3,4}$) to all varieties of minimal degree (which have been classified in algebraic geometry). It is well known that the Grassmannian $Gr_2(\mathbb{C}^l)\subset\mathbb{CP}^{\binom{l}{2}-1}$ ($l>4$) is not of minimal degree. Hence, there exist non-sos nonnegative quadratic forms on $Gr_2(\mathbb{C}^l)$ ($l>4$). Illustrating this result of Blekherman-Smith-Velasco, our examples (Corollary \ref{cor-BSV-ex} below) give explicit examples of such non-sos nonnegative quadratic forms on $Gr_2(\mathbb{C}^l)$. Recently, Bettiol-Kummer-Mendes \cite{BKM21} showed, among several important results, that the closed convex cone $P_{Gr_k(\mathbb{C}^l)}$  of nonnegative quadratic forms on $Gr_k(\mathbb{C}^l)$ ($2\leq k\leq l-2$ and $l\geq5$) is not a spectrahedral shadow. In their proof the existence of a non-sos nonnegative quadratic form on $Gr_k(\mathbb{C}^l)$ plays a key role.

Now let $\{E_1,\cdots,E_{m-1}\}$ be a Clifford algebra on $\mathbb{R}^l$, ($l=k\delta(m)\geq8$, $m\geq3$), whose associated \emph{psd} form $G_F^+\in P_{2l,4}$ is \emph{non-sos} in Theorem \ref{SOS-FKM}. Let $\varphi_{\alpha}\in\Lambda^2(\mathbb{R}^l)$ be the exterior $2$-forms defined by $\varphi_{\alpha}(u,v):=\langle E_\alpha u,v\rangle$, $\alpha=1,\cdots,m-1$. Note that we can also regard $\varphi_{\alpha}$ as linear functions on $\Lambda^2(\mathbb{R}^l)$ by setting $\varphi_{\alpha}(u\wedge v):=\varphi_{\alpha}(u,v)$. Then using Lagrange's identity we can rewrite (\ref{Rewritten}) as
\begin{equation}\label{nonsos-HL}
\frac{1}{8}G_F^+(x)=|u\wedge v|^2-\sum_{\alpha=1}^{m-1}(\varphi_{\alpha}(u\wedge v))^2\in P_{2l,4}\setminus \Sigma_{2l,4}, \quad x=(u,v)\in\mathbb{R}^l\times\mathbb{R}^l.
\end{equation}
It provides the following examples.
\begin{cor}\label{cor-BSV-ex}
Associated to each \emph{non-sos} \emph{psd} form $G_F^+\in P_{2l,4}$ in Theorem $\ref{SOS-FKM}$, $$\Phi_F(\omega):=|\omega|^2-\sum_{\alpha=1}^{m-1}(\varphi_{\alpha}(\omega))^2, \quad \omega\in\Lambda^2(\mathbb{R}^l),$$ is a \emph{non-sos} \emph{psd} quadratic form on $Gr_2(\mathbb{C}^l)\subset\mathbb{CP}^{\binom{l}{2}-1}$.
\end{cor}
\begin{proof}
Otherwise, suppose there were linear functions $\psi_1,\cdots,\psi_r$ (exterior $2$-forms in $\Lambda^2(\mathbb{R}^l)$)  such that
$$\Phi_F(\omega)=\sum_{i=1}^r(\psi_i(\omega))^2+P(\omega),$$
for some quadratic form $P$ in the span of the Pl\"{u}cker relations. Restricting $\Phi$ to $u\wedge v\in Gr_2(\mathbb{R}^l)$, we would get
$$|u\wedge v|^2-\sum_{\alpha=1}^{m-1}(\varphi_{\alpha}(u\wedge v))^2=\sum_{i=1}^r(\psi_i(u\wedge v))^2\in \Sigma_{2l,4}\; ,$$
a contradiction to (\ref{nonsos-HL}).
\end{proof}

Furthermore, these \emph{non-sos} forms also provide counterexamples to a generalized version of the Harvey-Lawson \emph{sos} problem (\cite[Question 6.5]{HL82}).
\begin{prob}[Harvey-Lawson \emph{sos} Problem]
Given a $p$-form $\varphi\in\Lambda^p(\mathbb{R}^l)$ such that
   $$\Phi(u_1,\cdots,u_p):=|u_1\wedge\cdots\wedge u_p|^2-\varphi(u_1,\cdots, u_p)^2, \quad u_1,\cdots,u_p\in\mathbb{R}^l,$$
  is nonnegative with nontrivial zeroes $($i.e., $\varphi$ has comass one$)$,
 are there exterior $p$-forms $\psi_1,\cdots,\psi_r$ such that
  the following \emph{sos} expression holds $$\Phi(u_1,\cdots,u_p)=\sum_{i=1}^r\psi_i(u_1,\cdots,u_p)^2 ~?$$
  \end{prob}
  In other words, it asks whether all the \emph{psd} quadratic forms $\Phi(\omega)=|\omega|^2-\varphi(\omega)^2$ are \emph{sos} on $Gr_p(\mathbb{C}^l)\subset\mathbb{CP}^{\binom{l}{p}-1}$. Thus Corollary \ref{cor-BSV-ex} shows that for $p=2$, this \emph{sos} problem cannot be generalized to ask for a \emph{sos} expression of \emph{psd} quadratic forms with the form $\Phi(\omega)=|\omega|^2-\sum_{\alpha=1}^{m}(\varphi_{\alpha}(\omega))^2$ when $m\geq2$.

\subsection{Zeros of the \emph{non-sos} \emph{psd} forms $G_F^+$}
To conclude this section, we would like to discuss further on the question of Solomon (\cite{So92}) as to whether both focal submanifolds
of an isoparametric hypersurface with $g=4$ in a unit sphere
$$M_{\pm}=(G_F^{\mp})^{-1}(0)\cap\mathbb{S}^{n-1}$$
 are quadratic.

The focal submanifolds $M_+$ are already quadratic (see (\ref{g422}, \ref{g445}, \ref{def-M+})). Still this question is important because of his result: \textit{A quadratic form vanishing on one focal submanifold is an eigenfunction of the minimal isoparametric hypersurface $M$ corresponding to the second known eigenvalue $2\dim M$ of the Laplacian on $M$}. A well known conjecture of Yau asserts that the first eigenvalue is $\dim M$ for a closed embedded minimal hypersurface $M$ in a unit sphere, which was proved in the isoparametric case by Tang and Yan \cite{TY13}. As we have shown in Sections \ref{secg42245} and \ref{secg4OTFKM}, $M_-$ in the exceptional two classes $(2,2)$, $(5,4)$ and the OT-FKM type with $m\equiv0~(mod~4)$ of the definite class are not quadratic, by proving that there exist no nonzero quadratic forms vanishing on them. In other classes $M_-$ may admit non-trivial quadratic forms vanishing on them. For example, for those OT-FKM type whose Clifford system $\{P_0,\cdots,P_m\}$ can be extended to a Clifford system $\{P_0,\cdots,P_m, P_{m+1}\}$ (there are many such classes, e.g., when $m\equiv 3,5,6,7~(mod~8)$ or $m\equiv0~ (mod~4)$ of the indefinite class), it is not difficult to show that the extended quadratic form
 $$P_{m+1}(x):=\langle P_{m+1}x,x\rangle$$
  vanishes on $M_-$. However, it seems probable that for all the classes with \emph{non-sos} $G_F^+$ in Table \ref{table-2}, $M_-$ is not a quadratic variety. If so, it
 would give a complete answer to the question of Solomon.

\begin{ack}
We thank sincerely Q.S.Chi, R.Miyaoka, B.Reznick, B.Solomon and G.Thorbergsson for their interest. We also want to thank anonymous referees for reading the manuscript very carefully and making a number of valuable comments.

\end{ack}

\end{document}